\documentclass[10pt]{amsart} 

\usepackage{amssymb,amsbsy,amsmath,amsfonts,amscd,latexsym,enumerate,mathrsfs,bm}

\newtheorem{theo}{Theorem}[section] 
\newtheorem{prop}[theo]{Proposition}

\newtheorem{lem}[theo]{Lemma} 
\newtheorem{cor}[theo]{Corollary}
\newtheorem{rem}[theo]{Remark}

\newtheorem{remark}[theo]{Remark}
\theoremstyle{definition}
\newtheorem{defi}[theo]{Definition}
 
\numberwithin{equation}{section} 
\def\tsp{{\,}^t\!}
 
\newcommand{\C}{\mathbb C} 
\newcommand{\R}{\mathbb R} 
\newcommand{\N}{\mathbb N} 
\newcommand{\B}{\mathbb B}

\begin{document}

\begin{abstract}
We construct a family of small analytic discs attached to Levi non-degenerate hypersurfaces in $\C^{n+1}$, which is globally biholomorphically invariant. We then apply this technique 
to study unique determination problems along Levi non-degenerate hypersurfaces that are merely of class $\mathcal{C}^4$. This method gives 2-jet determination results for germs of biholomorphisms, CR diffeomorphisms, as well as in the almost complex setting. 
\end{abstract} 

\title{Stationary holomorphic discs and finite jet determination problems} 

\author{Florian Bertrand and L\'ea Blanc-Centi}

\subjclass[2000]{}
\thanks{Research of the first author was supported by Austrian Science Fund FWF grants AY0037721 and M1461-N25.}
\maketitle

\section*{Introduction}

For many geometric structures, the automorphisms depend only on a finite number of parameters.
For example, an isometry $\varphi:M\to M$ of a connected Riemannian manifold $M$ is uniquely determined by its value $\varphi(p)$ and its differential $d\varphi_p$ at any given point $p\in M$, that is, by its 1-jet $j_p^1\varphi$. In complex analysis, biholomorphic automorphisms of a bounded domain in $\C^n$ are uniquely determined by their 1-jet at any given point inside the domain \cite{Ca}. 
What about the analogous statement if the point is taken on the boundary of the domain? It is a well-known fact that biholomorphic automorphisms of the unit ball $\B^n\subset\C^n$ extend holomorphically through the boundary and are uniquely determined by their 2-jets, but not by their 1-jets, at any given boundary point.

This problem is related with the local equivalence problem for boundaries of domains, or more generally for real hypersurfaces, which was started on in complex dimension 2 by H. Poincar\'e. 
Two real hypersurfaces $\Gamma$ and $\Gamma'$ are said to be {\it locally equivalent} at $p\in \Gamma$ and $p'\in\Gamma'$, respectively, if there exist a neighborhood $V$ of $p$ and $V'$ of $p'$ and a biholomorphic mapping $F:V\to V'$ so that $F(V\cap \Gamma)=V'\cap \Gamma'$. Local equivalence is a very restrictive condition, since it leads to an overdeterminated inhomogeneous system of partial differential equations for the biholomorphic mapping. If such an $F$ exists, then how unique is it? To answer this question, one has to look for sufficiently many invariants attached to the hypersurfaces. 

\bigskip

H. Poincar\'e initiated a study of invariants of a real hypersurface by looking at relations between the Taylor series coefficients of a defining function and the Taylor series of a transformed equation, in terms of the coefficients of the Taylor expansion of a local biholomorphic change of variable. This process of finding invariants from the power series expansion point of view was carried out much later in a significant manner by J.K. Moser
for Levi non-degenerate hypersurfaces. The following local version of H. Cartan's uniqueness theorem follows from the classical results of \cite{ECa,Ta1,Ta2,CM}: 

\begin{theo}\label{thm0}
Let $\Gamma$ be a real-analytic hypersurface through a point $p$ in $\C^n$ with non-degenerate Levi form and let $F_j:U_j\to V_j$, $j=1,2$, be two biholomorphic maps, where $U_j,V_j$ are open subsets in $\C^n$, $p\in U_j$ and $F_j(\Gamma\cap U_j)\subset \Gamma$ for any $j$. Then, if $F_1$and $F_2$ have the same 
2-jet at $p$, they coincide in a neighborhood of $p$.
\end{theo}
\noindent Theorem \ref{thm0} becomes false without any hypothesis on the Levi form, as one can see by considering the hyperplane $\Im m z_0=0$ in $\C^{n+1}$, whose automorphism group at 0 is infinite dimensional: every $F(z_0,z_\alpha)=(z_0,z_\alpha+f(z_0))$, for any holomorphic mapping $f$, is a biholomorphism preserving the hyperplane and fixing the origin. Note that in $\C^2$, Levi-flat hypersurfaces are the only ones (among real-analytic hypersurfaces) for which local biholomorphisms are not uniquely determined by their jets of any order \cite{ELZ}. 

Finite jet determination  of holomorphic mappings sending one real submanifold to another has attracted much attention in recent years, as related in the survey articles \cite{BER6} and \cite{Za}. After Theorem \ref{thm0}, other situations were investigated: 
finitely non-degenerate hypersurfaces \cite{BER1,Z}; hypersurfaces of finite type in $\C^2$ \cite{ELZ,KM}; 
 CR analogues \cite{Han1,Han2,BER2,BER5}.
 Note that it is always either assumed that $\Gamma$ and $\Gamma'$ are real-analytic, 
 or the conclusion is that the formal power series of the mapping $F$ is determined by a finite jet. The only finite jet determination result which applies to merely smooth hypersurfaces and smooth mappings is due to \cite{E} for smooth $k_0$-non-degenerate hypersurfaces: in the particular case $k_0=1$ (corresponding to Levi non-degenerate hypersurfaces), it states that
\textit{any holomorphic mapping which is defined locally on one side of a smooth, Levi non-degenerate real hypersurface $\Gamma\subset\C^n$ and extends smoothly to $\Gamma$, sending $\Gamma$ diffeomorphically into another smooth real hypersurface $\Gamma'\subset\C^n$, is completely determined by its 2-jet at a point $p\in \Gamma$}.

\bigskip

The aim of the present paper is to enter the finite jet determination problem with a more geometrical approach. The seminal approach of \cite{CM} is based on finding analytic invariants for real-analytic hypersurfaces (``normal form''). Here we propose to make use of some well-known invariants for a (sufficiently smooth) real hypersurface, namely the stationary holomorphic discs, that is, discs glued to the hypersurface and satisfying some differential condition at the boundary. These particular holomorphic discs were first introduced by L. Lempert \cite{L1}, as the complex geodesics for the Kobayashi metric for a strongly convex domain. They are involved in the construction of a ``Riemann map'' for various domains \cite{L1,ST} (see also \cite{CGS} for an almost complex version): roughly, the parametrization of this special family of holomorphic discs $h$ by $h\mapsto(h(0),h'(0))$ leads to a similar description for biholomorphic mappings. 

The point is to find an analogue at the boundary. A basic observation is that stationary holomorphic discs glued to the unit sphere are parametrized by their 1-jet at 0, but also by their 2-jet at any boundary point. Based on the fact that this is still true if one replaces the sphere by any non-degenerate hyperquadric,  we take this as our model case. 
The general situation of real hypersurfaces will be reduced to a perturbation of this model case. 
With this new method, we obtain an improvement of Theorem \ref{thm0} regarding the smoothness, since we do not need real analyticity, nor $\mathcal{C}^\infty$-smoothness. More precisely, we get:

\begin{theo}\label{thmjet}
Let $\Gamma$, $\Gamma'$ be two real hypersurfaces of class
$\mathcal{C}^4$ in
$\C^{n+1}$, and $p\in\Gamma$. Assume $\Gamma$ is Levi non-degenerate at
$p$. Then the germs at $p$ of biholomorphisms $F$ such that $F(\Gamma)=\Gamma'$ are uniquely determined by their 2-jet at $p$.
\end{theo}

\bigskip

In fact, the proof of Theorem \ref{thmjet} only uses the fact that stationary discs glued to $\Gamma$ are taken by $F$ to stationary discs glued to $\Gamma'$. 
 If the hypersurfaces are strictly pseudoconvex, the real hypothesis is thus that the holomorphic map $(F,dF^{-1})$ is defined locally on one side of $\Gamma$ and extends 
 $\mathcal{C}^1$-smoothly to $\Gamma$ (see remark \ref{cote}). Since a germ of a CR diffeomorphism between two strictly pseudoconvex hypersurfaces   
 admits a local one-sided extension, we obtain the following result:   
\begin{theo}\label{thmcr}
Let $\Gamma,\ \Gamma'\subset \C^{n+1}$ be two strictly pseudoconvex hypersurfaces of class $\mathcal{C}^4$, and let 
$p\in\Gamma$. Then the germs at $p$ of CR diffeomorphisms of class $\mathcal{C}^3$ between $\Gamma$ and $\Gamma'$ are uniquely determined by their 
2-jet at $p$.
\end{theo} 
Another situation where this condition occurs near any boundary point  is provided by biholomorphisms, or more generally a proper holomorphic maps, between two bounded strictly pseudoconvex domains with $\mathcal{C}^2$ boundary. So we get the following boundary version of the classical H. Cartan's uniqueness Theorem:
\begin{theo}\label{thmjetstrict}
Let $\Omega,\ \Omega'\subset \C^{n+1}$ be two bounded strictly pseudoconvex domains whose boundaries are of class 
$\mathcal{C}^4$, and let $p\in\partial\Omega$. If $F_1$ and $F_2$ are two proper maps from $\Omega$ onto $\Omega'$ with the same 2-jet at $p$, they coincide.
\end{theo}
\noindent This statement is in the following of previous results on holomorphic auto-applications $F$ on a domain $D$. Let us recall that if $D$ is a connected, smoothly bounded, strongly pseudoconvex domain, and $F(z)=z+O(|z-p|^4)$, then $F=id$ \cite{BK}. If $D$ is strongly convex, the assumption reduces to $F(z)=z+o(|z-p|^2)$ \cite{Hu}. Finally, see \cite{LM} for some finite jet determination results for proper maps, without convexity but with real-analytic boundary regularity hypothesis on the domains.

\bigskip

One of the main interests of a method using holomorphic discs is that it usually extends to the almost complex setting.
In the last section, we explain how to obtain almost complex analogues of Theorems \ref{thmjet} and \ref{thmjetstrict}. 
Our most significant result  in the almost complex case is the following:
\begin{theo}\label{theopC}
Let $J$ and $J'$ be two almost complex structures of class $\mathcal{C}^{3}$ 
defined in $\R^{4}$. 
Let $\Gamma$, $\Gamma'$ be two real hypersurfaces of class $\mathcal{C}^4$. Assume $\Gamma$ is $J$-Levi non-degenerate at
$p$. Then the germs at $p$ of $(J,J')$-biholomorphisms $F$ such that $F(\Gamma)=\Gamma'$ are uniquely determined by their 2-jet at $p$.
\end{theo}

\bigskip
Let us finally notice that finite jet determination problem is relevant in higher codimension \cite{Belo,L}. We recall that if $\Gamma$ is a generic real-analytic submanifold of any codimension with non-degenerate Levi form at $p$, then its local biholomorphisms are uniquely determined by their 2-jet at $p$. Since stationary discs play a role not only for hypersurfaces \cite{Tumanov}, our method could give results even in this case.

\bigskip

This paper is organized as follows. In the first section, we recall some classical facts about Levi non-degenerate hypersurfaces and stationary discs. Section 2 and Section 3 are devoted to the parametrization of the family of stationary discs whose boundaries pass through a prescribed point, respectively for the model case (non-degenerate hyperquadric) and for small perturbations. We then prove Theorem \ref{thmjet} and Theorem \ref{thmcr} in Section 4. Finally we give in Section 5 the extension to the almost complex situation.

\vskip 0,1cm
\noindent  {\it Acknowledgments.} The authors would like to thank B. Lamel for getting them interested in the problem and for stimulating  discussions and helpful remarks.

\section{A family of biholomorphic invariants}

\subsection{Levi non-degenerate hypersurfaces}
For $z\in\C^{n+1}$ we write $z=(z_0,z_\alpha)$ where
$z_\alpha=(z_1,\hdots,z_n)$. As usual $x_j=\Re e z_j$ and $y_j=\Im m z_j$. We
deal with real hypersurfaces in $\C^{n+1}$:
$$\Gamma=\{(x_0,\hdots,y_n)\in\C^{n+1}\ |\ \rho(x_0,\hdots
y_n)=0\}=\{z\in\C^{n+1}\ |\ \rho(z)=0\}$$ where 
$\rho$ is a $\mathcal{C}^4$ defining function of $\Gamma$,
that is $\rho$ is of class 
$\mathcal{C}^4$ and 
$d\rho$ does not vanish on
$\Gamma$. Then the tangent space at $p\in\Gamma$ is $T_p\Gamma=\mathrm{Ker}\,d\rho_p$, and the complex tangent space at $p$ (\textsl{i.e.} the biggest complex subspace of $T_p\Gamma$) is $T_p^\C\Gamma=T_p\Gamma\cap iT_p\Gamma$. 
\begin{defi} 
The hypersurface $\Gamma$ is {\it Levi non-degenerate} at point $p\in\Gamma$ if the restriction to $T_p^\C\Gamma$ of the Hermitian form $\sum_{0\le i,j\le n}\frac{\partial^2\rho}{\partial\bar{z}_j\partial z_i}\bar{z}_jz_i$ is non-degenerate.
\end{defi}
The model cases are non-degenerate hyperquadrics $x_0=\tsp \overline{z}_\alpha Az_\alpha$ with $A$ an invertible Hermitian matrix, which are everywhere Levi non-degenerate.  

Pick $p\in\Gamma$. Up to a linear transformation, then using the implicit function theorem, we can
assume that $p=0$ and the tangent space to $\Gamma$ at 0 is $x_0=0$,
hence $\Gamma$ is given by the equation
$$x_0=\left(\begin{array}{c}{\,}_{\rm{real\ quadratic\ terms\ in}}\\  
    y_0,x_1,\hdots,y_n\end{array}\right)+ O(|(y_0, z_\alpha)|^3).$$
Instead of grouping terms together according to their degree, we regroup them
by weight (see \cite{CM}): a weight 2 is assigned to $z_0$ and
a weight 1 is assigned to $z_\alpha$. We get the second member of the previous equality
decomposed in three parts:
\begin{itemize}
\item terms of weight 2: a real quadratic form in $z_\alpha$, that is
  $q(z_\alpha)+\overline{q(z_\alpha)}+(Hermitian\ form\ in\ z_\alpha)$ with $q$ being
  a complex quadratic form;
\item terms of weight at least 3: linear combination of $y_0^2$, $y_0x_j$ and
  $y_0y_j$ ($1\le j\le n)$;
\item $O(|(y_0, z_\alpha)|^3)$, which would consist of terms of weight at
  least 3 if the hypersurface were real analytic.
\end{itemize}
After the local change of coordinates $z_0=z_0'-2q(z_\alpha),\
z_\alpha=z_\alpha'$, we obtain that the hypersurface $\Gamma$ is given in a
neighborhood of 0 by the defining function

\begin{equation}\label{normalform}
\rho(z)=x_0-\tsp \overline{z}_\alpha Az_\alpha+b_0y_0^2+\sum_{j=1}^n(b_jz_j+\bar{b}_j\bar{z}_j)y_0+O(|(y_0,z_\alpha)|^3).
\end{equation}
This depends only on $\Gamma$ (and on the point $p$). Notice that
$b_0=\frac{1}{2}\frac{\partial^2\rho}{\partial y_0^2}(0)$,
$b_j=\frac{\partial^2\rho}{\partial y_0\partial z_j}(0)$ and $A$ is the Hermitian matrix
$A=-\left(\frac{\partial^2\rho}{\partial\bar{z}_j\partial z_i}(0)\right)_{1\le
  i,j\le n}$.
Hence $\Gamma$ is Levi non-degenerate at 0 iff the matrix $A$ is invertible. In this case, $\Gamma$ is also Levi non-degenerate in a neighborhood of 0.

\bigskip

Afterwards, we assume that functions $\rho$ are under
the normal form (\ref{normalform}), and set $\Gamma^{\rho}=\{z\ |\ \rho(z)=0\}$: locally, $\Gamma^\rho$ ``looks
like'' the hyperquadric of equation $x_0=\tsp \overline{z}_\alpha Az_\alpha$.
The Hermitian matrix $A$ is not a biholomorphic invariant of $\Gamma^\rho$, but its signature
$$(number\ of\ positive\ eigenvalues,\ number\ of\ negative\
eigenvalues)$$ 
is a biholomorphic invariant up to the sign. In particular, if $\Gamma^\rho$ is Levi non-degenerate at 0 and $\Gamma^{\rho'}$ is locally biholomorphic to $\Gamma$ in a neighborhood of 0, then $\Gamma^{\rho'}$ is also Levi non-degenerate at 0 (and the signature of $A'$ is $\pm\,\mathrm{sgnt}(A)$).

\begin{defi}
For $p\in \Gamma$, let   
$N_p^*\Gamma:=\{\phi\in T_p^*\C^{n+1} | \Re e \phi_{|T_p\Gamma}=0\}$ 
be the is a real line generated by $\partial\rho_p$.\\ 
The {\it conormal bundle} $N^*\Gamma$ of $\Gamma$ is the bundle over $\Gamma$ whose fiber at $p\in\Gamma$ is $N_p^*\Gamma$. 
\end{defi}

In particular, $N^*\Gamma$ is a real submanifold of dimension $2n+2$ of the complex manifold $T^*\C^{n+1}$.
 We will need the 
following characterization due to A. Tumanov:
\begin{prop}\label{propco}{\cite{Tumanov}}
A real hypersurface $\Gamma \subset \C^{n+1}$ is Levi non-degenerate if and only if its 
conormal bundle $N^*\Gamma$ is totally real out of the zero section.
\end{prop}
Let us recall that a submanifold of a complex manifold is {\it totally real} at some point if its complex tangent space at this point is trivial.

\subsection{Stationary discs} \label{subsection-stationary}
A {\it holomorphic disc} $h$ is a holomorphic function on the open unit disc
$\Delta \subset \C$. When $h$ is continuous up to the 
boundary and $h(\partial\Delta)$ is included in some submanifold, we say that $h$ is {\it glued} to this submanifold.

The first useful case is when the submanifold is a hypersurface, for example the boundary of a domain $D$. Then holomorphic discs in $D$ glued to $\partial D$ are invariants of $\overline{D}$. The second interesting case is when the submanifold is maximally totally real, which gives for instance some properties of smoothness for the discs up to the boundary. In our case, looking at Levi non-degenerate hypersurfaces, we can take advantage of both situations by using Proposition \ref{propco}. This leads to the study of stationary discs.

\begin{defi}\label{defstat}
A holomorphic disc $h$ glued to the real hypersurface $\Gamma$ is {\it stationary} if there exists a holomorphic lift $\bm{h}=(h,g)$ of $h$ to the cotangent bundle $T^*\C^{n+1}$, continuous up to the boundary, such that $\forall\zeta \in\partial\Delta,\ \bm{h}(\zeta)\in\mathscr{N}\Gamma(\zeta)$
where
$$\mathscr{N}\Gamma(\zeta):=\{(z,\zeta w)\ |\ z\in\Gamma,\ w\in N_z^*\Gamma\setminus\{0\}\}.$$
The set of these lifted discs $\bm{h}=(h,g)$, with $h$ non-constant, is denoted by $\mathscr{S}(\Gamma)$.
\end{defi}

Note that if $\bm{h}=(h,\zeta h^*)\in\mathscr{S}(\Gamma)$ satisfies $h(1)=0$, then there exists $b\in\R^*$ such that $h^*(1)=(b,0,\hdots,0)$ since $\partial\rho_0=(\frac{1}{2},0,\hdots,0)$ (the defining equation of $\Gamma$ being in normal form). We set 
$$\mathscr{S}^*(\Gamma):=\{\bm{h}=(h,\zeta h^*)\in\mathscr{S}(\Gamma)\ |\ h(1)=0,\ h^*(1)=(1,0,\hdots,0)\}.$$

In local coordinates, Definition \ref{defstat} is equivalent to the existence of a continuous function 
$c:\partial\Delta\to\R^*$ such that $g(\zeta)=\zeta c(\zeta)d\rho_{h(\zeta)}$ on $\partial\Delta$ and 
 extends holomorphically to $\Delta$.
We will often set $g(\zeta)=\zeta h^*(\zeta)$: then $h^*$ is meromorphic with at most one pole of order one at 0, and 
$(h,h^*)(\partial\Delta)$ is included in the conormal bundle out of the zero section. We actually need to allow the fibered part $h^*$ to have a pole of order one at 0 (let us notice that holomorphic discs glued to the unit sphere have no holomorphic lift to the cornormal bundle). 

\medskip

If $\Gamma=\Gamma^\rho$ is Levi non-degenerate, then by Proposition \ref{propco} $\mathscr{N}\Gamma(\zeta)$ is totally real for all $\zeta\in\partial\Delta$.
We say that the holomorphic disc $\bm{h}$ is {\it attached to the totally real fibration} $\mathscr{N}\Gamma=\{\mathscr{N}\Gamma(\zeta)\}$ (see section \ref{sectionGlob} for the definition of a totally real fibration over $\partial\Delta$). 
\bigskip

Notice that, in light of Proposition \ref{propco}, the regularity results
known for holomorphic discs glued to a totally real submanifold \cite{Chirka}
apply to the map $\bm{h}$ if $\Gamma$ is Levi non-degenerate. Hence the elements of $\mathscr{S}(\Gamma)$ (and thus stationary discs 
and their regular lifts) inherit their regularity to the boundary from the H\"olderian regularity of $N^*\Gamma$: since $\Gamma$ is $\mathcal{C}
^{4}$, $N^*\Gamma$ is $\mathcal{C}^{2,\epsilon}$ for any $0<\epsilon<1$ and the elements in $\mathscr{S}(\Gamma)$ are in $\mathrm{Hol}
(\Delta)\cap\mathcal{C}^{2,\epsilon}(\bar{\Delta})$. This means that every $\bm{h}\in\mathscr{S}(\Gamma)$ satisfies $\bm{h}\in\mathcal{C}
^{2,\epsilon}(\bar{\Delta},T^*\C^{n+1})$ or, equivalently, $\bm{h}_{|\partial\Delta}\in\mathcal{C}^{2,\epsilon}(\partial\Delta,T^*\C^{n+1})$. The 
spaces $\mathcal{C}^{k,\epsilon}(\partial\Delta)$, $0<\epsilon<1$, $k\in\N$ are equipped with their usual norm:
$$\|\bm{h}\|_{\mathcal{C}^{k,\epsilon}(\partial\Delta)}=\sum_{l=0}^{k}\|\bm{h}^{(l)}\|_\infty+
\underset{\zeta\not=\eta\in\partial\Delta}{\mathrm{sup}}\frac{\|\bm{h}(\zeta)-\bm{h}(\eta)\|}{|\zeta-\eta|^\epsilon},$$
where $\|\bm{h}^{(l)}\|_\infty:=\underset{\partial\Delta}{\mathrm{max}}\|\bm{h}^{(l)}\|$.

The interest for stationary discs comes from the fact they are biholomorphic invariants. More precisely, if $F$ is a biholomorphism such that $F(\Gamma)\subset\Gamma'$, and $h$ is a stationary disc glued to $\Gamma$, then $F\circ h$ is a stationary disc glued to $\Gamma'$: actually, if $h^*$ is a regular lift of $h$, then $h^*\cdot (d{F}_{h})^{-1}$ is a regular lift of $F\circ h$. We denote, for $\bm{h}=(h,\zeta h^*)\in\mathscr{S}(\Gamma)$:
$$F_*\bm{h}:\zeta\mapsto\left(F(h(\zeta)),\zeta h^*(\zeta)(d{F}_{h(\zeta)})^{-1}\right)$$
and get $F_*\bm{h}\in\mathscr{S}(\Gamma')$. 

\begin{rem}
Assume $F$ is the identity up to order one, that is, $F(0)=0$ and $dF_0=\mathrm{id}$. Then for every $\bm{h}\in\mathscr{S}^*(\Gamma)$,  we get $F_*\bm{h}\in\mathscr{S}^*(\Gamma')$.
\end{rem}

\section{The model situation: biholomorphisms between two hyperquadrics}

In this section, we study the following situation. Assume $F$ is a
biholomorphism fixing $0$ between two hyperquadrics $Q^A=\{r^A=0\}$ and
$Q^{A'}=\{r^{A'}=0\}$ in $\C^{n+1}$:
$$r^A(z):=\Re e z_0-\tsp \overline{z}_\alpha Az_\alpha$$
$$r^{A'}(z):=\Re e z_0-\tsp \overline{z}_\alpha A'z_\alpha,$$
where the Hermitian matrices $A$ and $A'$ are assumed to be invertible. We first prove
that such biholomorphisms are determined by their 2-jet at 0. It suffices to
consider the case when $Q^{A}=Q^{A'}:=Q$ and $F$ is equal to the identity up to
order two. Even if the automorphisms of hyperquadrics are easy to determine
(see for instance \cite{CM}), we will not use their explicit expression but only the fact that the family of 
stationary discs is a global biholomorphic invariant. This method will be generalized in the following sections.

\subsection{Stationary discs glued to $Q$}
Stationary discs glued to a non-degenerate hyperquadric are explicitly
known, and they are actually uniquely determined by their 2-jet at some boundary
point. Let us recall that, based on Propositions 2.1 and 2.3 from \cite{BL}, we have the following explicit expression for
the elements of $\mathscr{S}(Q)$:

\begin{prop} 
The elements $\bm{h}=(h,\zeta h^*)\in\mathscr{S}(Q)$ are exactly under the form
$$h(\zeta)=\left(\tsp\bar{v}Av+2\tsp\bar{v}Aw\,\frac{\zeta}{1-a\zeta}+\frac{\tsp\bar{w}Aw}{1-|a|^2}\,\frac{1+a\zeta}{1-a\zeta}+iy_0,v+w\,\frac{\zeta}{1-a\zeta}\right)$$
$$\zeta h^*(\zeta)=b(\zeta-\overline{a})(1-a\zeta)\times(1/2,-\tsp\overline{h_\alpha(\zeta)}A)$$
where $a\in\Delta$, $v\in\C^n$, $w\in\C^n\setminus\{0\}$, $y_0\in\R$, and $b\in\R^*$. 
Moreover, the  map $(a,v,w,y_0,b)\mapsto \bm{h}$ is a smooth parametrization of $\mathscr{S}(Q)$,
which makes $\mathscr{S}(Q)$ smoothly diffeomorphic to $\Delta\times\C^n\times(\C^n\setminus\{0\})\times\R\times\R^*$.
\end{prop}

\begin{cor}\label{submersion}\label{corsub}
The maps $\bm{h}\mapsto{h}(1)\in Q$ and $\bm{h}\mapsto\bm{h}(1)\in N^*Q$ defined on $\mathscr{S}(Q)$ are submersions.
\end{cor}

\begin{proof}
According to the previous parametrization of $\mathscr{S}(Q)$, 
$$h(1)=\left(\tsp\bar{v}Av+2\tsp\bar{v}Aw\,\frac{1}{1-a}+\frac{\tsp\bar{w}Aw}{1-|a|^2}\,\frac{1+a}{1-a}+iy_0,v+w\,\frac{1}{1-a}\right)$$
$$h^*(1)=b|1-a|^2\times(1/2,-\tsp\overline{h_\alpha(1)}A)$$
Assume $(a,v,w,y_0,b)\in\Delta\times\C^n\times(\C^n\setminus\{0\})\times\R\times\R^*$, and let $\bm{h}=(h,\zeta h^*)$ be the corresponding element in $\mathscr{S}(Q)$. 
Notice that 
\begin{equation}\label{equationConormal}
(z,t)\in N^*Q\ \Longleftrightarrow\ z\in Q,\ t_0\in\R,\ t_\alpha+2t_0\tsp\overline{z}_\alpha A=0
\end{equation}
hence
$$(\mu,\tau)\in T_{(z,t)}(N^*Q)\ \Longleftrightarrow\ 
\left\{\begin{array}{l}
\mu\in T_zQ\\
\tau_0\in\R\\
\tau_\alpha+2\tau_0\tsp\overline{z}_\alpha A+2t_0\tsp\overline{\mu_\alpha}A=0
\end{array}\right.$$

We first  consider the map
$\phi^1:(a,v,w,y_0,b)\mapsto h(1)\in Q$.
For every $\mu\in T_{h(1)}Q$, that is such that $\Re e\,\mu_0=\tsp\overline{h_\alpha(1)}A\mu_\alpha+\tsp\overline{\mu_\alpha}Ah_\alpha(1)$, we are looking for some
$(a',v',w',y'_0,b')\in\C\times\C^n\times\C^n\times\R\times\R$ such that $d\phi^1_{(a,v,w,y_0,b)}(a',v',w',y'_0,b')=\mu$.
The $n$ last components of this equality give the equation
$$v'+\frac{1}{1-a}w'+w\frac{1}{(1-a)^2}a'=\mu_\alpha$$
and we can choose $v'=\mu_\alpha$, $w'=0$ and $a'=0$. Since by construction $d\phi^1_{(a,v,w,y_0,b)}(a',v',w',y'_0,b') \in T_{h(1)}Q$, the first component applied to these values only corresponds to the equation
$$\Im m\left(2\tsp\overline{v'}Aw\frac{1}{1-a}\right)+y'_0=\Im m\,\mu_0$$
thus we can find a convenient $y'_0$.

Consider now the map
$\phi^2:(a,v,w,y_0,b)\mapsto h^*(1)$. Suppose $(\mu,\tau)\in T_{(h(1),h^*(1))}(N^*Q)$ and choose $v'=\mu_\alpha$, $w'=0$, $a'=0$ and $y'_0$ the convenient value found previously.
 We are looking for some $b'\in\R$ such that $d\phi^2_{(a,v,w,y_0,b)}(0,\mu_\alpha,0,y'_0,b')=(\mu, \tau)$. Then $b'=\frac{2}{|1-a|^2}\,\tau_0$ is convenient.
\end{proof}

The condition $h(1)=0,\ h^*(1)=(1,0,\hdots,0)$ is
equivalent to
$$\left\{\begin{array}{l}
w=-(1-a)v\\
iy_0=-\frac{a-\bar{a}}{1-|a|^2}\tsp\bar{v}Av\\
\frac{b|1-a|^2}{2}=1
\end{array}\right.$$
and a straightforward computation gives that the elements of $\mathscr{S}^*(Q)$ are exactly of the following form: 
\begin{eqnarray}
h(\zeta)=\frac{1-\zeta}{1-a\zeta}\left(2\frac{1-a}{1-|a|^2}\tsp\bar{v}Av,v\right)=\frac{1-\zeta}{1-a\zeta}h(0)\\
\zeta
h^*(\zeta)=\frac{2}{|1-a|^2}\,\left(\frac{(\zeta-\overline{a})(1-a\zeta)}{2},(1-\zeta)(1-a\zeta)\tsp\bar{v}A\right)
\label{lift}
\end{eqnarray}
where $a\in\Delta$ and $v\in\C^n\setminus\{0\}$.

\subsection{Parametrization}

\begin{prop}\label{2jet-discs}\ 
Let $Q$ be a non-degenerate hyperquadric in $\C^{n+1}$. Then 
$\mathscr{S}^*(Q)$ is a (2n+2)-real parameter family. Moreover:
\begin{enumerate}[i)]
\item the elements $\bm{h}=(h,g)\in\mathscr{S}^*(Q)$ are smooth up to the boundary;
\item the map $\bm{h}\mapsto h(0)$ is a smooth diffeomorphism from $\{\bm{h}=(h,g)\in\mathscr{S}^*(Q)\
  |\ \tsp\overline{h_\alpha(0)}Ah_\alpha(0)\not=0\}$ onto the open set
  $\{(\gamma\tsp\bar{v}Av,v)\ |\ v\in\C^n,\ \tsp\bar{v}Av\not=0,\ \Re e(\gamma)>1\}$;
\item the map $\bm{h}\mapsto(h_\alpha'(1),h'_0(1)g'_0(1))$ is a smooth diffeomorphism from 
$\{\bm{h}\in\mathscr{S}^*(Q)\ |\ \tsp\overline{h_\alpha(0)}Ah_\alpha(0)\not=0\}$ onto its image.
\end{enumerate}
\end{prop}
\noindent Item $iii)$ implies that the map $\bm{h}\mapsto\bm{h}'(1)$ defined on $\{\bm{h}\in\mathscr{S}^*(Q)\ |\ \tsp\overline{h_\alpha(0)}Ah_\alpha(0)\not=0\}$ is one-to-one.

\begin{remark}
We could also check that the map $\bm{h}\mapsto (h_\alpha'(1),\tsp\overline{h_\alpha'(1)}h_\alpha''(1))$ defined on $\mathscr{S}^*(Q)$ is a smooth diffeomorphism onto its image, hence non-constant stationary discs $h$ glued to $Q$ such that $h(1)=0$ are uniquely determined by their 2-jet at point 1.
\end{remark}

\begin{proof}
According to Corollary \ref{corsub}, $\mathscr{S}^*(Q)$ is a submanifold of real dimension $(4n+4)-(2n+2)$ of $\left(\mathcal{C}^{1,\epsilon}(\partial\Delta)\right)^{2n+2}$, given by the one-to-one parametrization $\Delta\times(\C^n\setminus\{0\})\ni (a,v)\mapsto \bm{h}$.

Using this parametrization of $\mathscr{S}^*(Q)$, we are able to localize the centers of such discs by looking at the map
$(a,v)\mapsto\left(2\frac{1-a}{1-|a|^2}\tsp\bar{v}Av,v\right)$.
The function $a\mapsto \frac{2(1-a)}{1-|a|^2}$ is a bijection from the unit disc $\Delta$ onto the half plane 
$\{\Re e \zeta> 1 \}$ whose inverse is given by $\zeta\mapsto \frac{2\zeta-\zeta^2}{|\zeta|^2}$.  
Since we consider discs such that the parameter $v=h_\alpha(0)$ satisfies $\tsp\bar{v}Av\neq 0$, this implies that the map 
is a smooth diffeomorphism.
As a direct consequence, we obtain $(ii)$.

To prove $(iii)$, let $\bm{h}=(h,g)\in\mathscr{S}^*(Q)$. We have 
$$h'(1)= \frac{-1}{1-a}\left(\frac{2(1-a)}{1-|a|^2}\tsp\overline{v}Av,v\right)=\left(\frac{-2}{1-|a|^2}\tsp\overline{v}Av,\frac{-1}{1-a}v\right),$$
$$g'(1)=\frac{2}{|1-a|^2}\left(\frac{1+|a|^2-2a}{2},(a-1)\tsp\overline{v}A\right)=\left(\frac{1+|a|^2-2a}{|1-a|^2},\frac{-2}{1-\overline{a}}\tsp\overline{v}A\right).$$
We have to consider the map
$\bm{h}\mapsto(h_\alpha'(1),h'_0(1)g'_0(1))$. Since $h_\alpha'(1)=\frac{-1}{1-a}h_\alpha(0)$ and we assume 
$\tsp\overline{h_\alpha(0)}Ah_\alpha(0)\not=0$, it is equivalent to look at
$$\bm{h}\mapsto\left(-h_\alpha'(1),\frac{-1}{2\tsp\overline{h_\alpha'(1)}Ah_\alpha'(1)}\,h'_0(1)g'_0(1)\right).$$
Using the parametrization of $h\in\mathscr{S}^*(Q)$ by $(a,v)$, 
we set for any $(a,v)\in \Delta\times(\C^n\setminus\{0\})$ such that $\tsp\overline{v}Av\not=0$,
$$\psi(a,v)=\left(\frac{1}{1-a}v,\frac{1+|a|^2-2a}{1-|a|^2}\right)\in\C^{n+1}.$$
Note that $\psi(a,v)=\left(\frac{1}{1-a}v,-1+\frac{2(1-a)}{1-|a|^2}\right)$: once more, the properties of the function $a\mapsto \frac{2(1-a)}{1-|a|^2}$ prove that $\psi$ is a smooth, one-to-one immersion. We conclude the proof by using of the inverse function theorem.

\end{proof}

Let us see how it leads to the
unique determination of $F$ by its 2-jet at the origin. Assume $F$ is a
biholomorphism of $\C^{n+1}$ such that $F(0)=0$ and $F(Q)\subset Q$, with the
same 2-jet than the identity. Let $z\in \Omega:=\{(\gamma\tsp\bar{v}Av,v)\ |\ v\in\C^n,\
\tsp\bar{v}Av\not=0,\ \Re e(\gamma)>1\}$: by $(ii)$ there exists a unique 
$\bm{h}=(h,g)\in\mathscr{S}^*(Q)$ such that $h(0)=z$. Since $F(0)=0$ and $dF_0=\mathrm{id}$, the disc $F_*\bm{h}$ is still in $\mathscr{S}^*(Q)$. Moreover $(F_*\bm{h})'(1)=\bm{h}'(1)$, so by $(iii)$
we get $F_*\bm{h}=\bm{h}$ and hence $F\circ h=h$, and $F(z)=z$. This is true for any $z$ in the open
set $\Omega$ so $F$ is equal to the identity. This can be resumed in the
following commuting diagram:

$$\begin{array}{cccc}
F:  &h(0)&\mapsto&F\circ h(0)\\
  &\updownarrow& \circlearrowleft&\updownarrow\\
 j_0^1(F,\tsp dF^{-1}): &\bm{h}'(1) & \mapsto& (F_*\bm{h})'(1). 
\end{array}$$
where $j_0^1(F,\tsp dF^{-1})$ denotes the $1$-jet at $0$ of $(F,\tsp dF^{-1})$. 
\begin{remark}\label{cote}
Note that if $Q$ is strictly pseudoconvex (\textsl{i.e.} the Hermitian matrix $A$ is positive definite), then non-constant discs $h$ glued to $Q$ remain on the same side of the hyperquadric. Indeed,  the map $-r\circ h$ is subharmonic and one can conclude by using the maximum principle.
Hence in this case the previous diagram is still valid if the map $(F,\tsp dF^{-1})$ is only defined on one side of $Q$, and of class $\mathcal{C}^1$ up to the boundary.
\end{remark}

\section{Discs attached to a perturbation of $\mathscr{N}Q$}

The aim of this section is to generalize Proposition \ref{2jet-discs} to the case of a small perturbation of the hyperquadric $Q$. 
The method consists in using a theorem of J. Globevnik \cite{Glob}:
 given a totally real fibration $\mathscr{E}$ and a disc $f$ attached to $\mathscr{E}$, and under the conditions that some integers depending on $\mathscr{E}$ and $f$ are non-negative, the holomorphic discs near $f$ attached to a small perturbation of $\mathscr{E}$ form a $\kappa$-parameter family, where $\kappa$ is the Maslov index of $\mathscr{E}$ along $f$. 
In our situation, 
 we choose $\mathscr{N}Q$ as the totally real fibration and we fix 
a disc $\bm{h} \in \mathscr{S}^*(Q)$. This will give a local description of $\mathscr{S}^*(\Gamma)$ for a small perturbation $\Gamma$   of $Q$.

\subsection{The result of J. Globevnik\label{sectionGlob}}
Let $0<\epsilon<1$. Consider the following situation:
\begin{itemize}
\item $\B\subset \C^{N}$ is an open ball centered at the origin and $\tilde{r}_1,\hdots,\tilde{r}_N$ are in $\mathcal{C}^{1,\epsilon}(\partial\Delta,\mathcal{C}^3(\B,\R))$
\item $f$ is a map of class $\mathcal{C}^{1,\epsilon}$ from $\partial\Delta$ to $\B$
\item for every $\zeta\in\partial\Delta$,
\begin{enumerate}[i)]
\item $\mathscr{E}(\zeta):=\{\omega\in\B | \tilde{r}_j(\zeta)(\omega)=0,\ 1\le j\le N\}$ is a maximal totally real submanifold in $\C^N$,
\item $f(\zeta)\in\mathscr{E}(\zeta)$,
\item $\partial _\omega\tilde{r}_1 \wedge \hdots \wedge \partial_\omega \tilde{r}_N$ does not vanish on $\partial\Delta\times\B$.
\end{enumerate}
\end{itemize}

\noindent Such a family $\mathscr{E}:=\{\mathscr{E}(\zeta)\}$ of manifolds is called a {\it totally real fibration} over $\partial \Delta$. A disc glued to a fixed totally real manifold ($\mathscr{E}$ is independent of $\zeta$) is a special case of a totally real fibration.
 
\bigskip

Denote by  $GL_N(\C)$ the group of all invertible $(N\times N)$ matrices with complex entries. Let $\zeta\in\partial \Delta$ and consider the 
matrix $G(\zeta):=\displaystyle \left(\frac{\partial \tilde{r}_i}{\partial\bar{z}_j}(f(\zeta))\right)_{i,j} \in GL_N(\C)$. For any $(N\times N)$ matrix $A(\zeta)$ 
whose columns span $T(\zeta):=T_{f(\zeta)}(\mathscr{E}(\zeta))$, any row of $G(\zeta)$ is orthogonal to any column of $A(\zeta)$:
$$\Re e(\overline{G(\zeta)}A(\zeta)=0\Longleftrightarrow G(\zeta)\overline{A(\zeta)}=-\overline{G(\zeta)}A(\zeta)\Longrightarrow A(\zeta)
\overline{A(\zeta)}^{-1}=-\overline{G(\zeta)}^{-1}G(\zeta).$$
Set $B(\zeta)=A(\zeta)\overline{A(\zeta)}^{-1}=-\overline{G(\zeta)}^{-1}G(\zeta)$ for all $\zeta\in\partial \Delta$. Hence the matrix $B(\zeta)$ depends only on $T(\zeta)$ and not on a particular choice of defining functions. 
Moreover, one can find a Birkhoff factorization of $B$ (see \cite{Birkhoff}), \textsl{i.e.} some continuous  functions $B^+:\bar{\Delta}\to GL_N(\C)$ and 
$B^-:(\C \cup \infty)\setminus\Delta\to GL_N(\C)$  such that 
$$\forall \zeta\in\partial \Delta,\ B(\zeta)=B^+(\zeta)\left(\begin{array}{ccc}\zeta^{\kappa_1}& &(0) \\ &\ddots& \\ (0)& &\zeta^{\kappa_{N}}\end{array}\right)B^-(\zeta)\,$$
where $B^+$ and $B^-$ are holomorphic on $\Delta$ and $\C \setminus \overline{\Delta}$ respectively.
The integers $\kappa_1\ge\hdots\ge\kappa_N$ do not depend on this factorization. They are called the {\em partial indices} of $B$ (see \cite{Vekua,CG} for more details) or the {\em partial indices of $\mathscr{E}$ along $f$}. 
The {\it Maslov index of $\mathscr{E}$ 
along $f$} is the sum $\sum_{1}^N\kappa_j$.

The following result was stated in \cite{Glob} for a fibration given by equations in 
$\mathcal{C}^{\epsilon}(\partial\Delta,\mathcal{C}^2(\B)^N)$, but the arguments remain valid for $\mathcal{C}^{1,\epsilon}(\partial\Delta,\mathcal{C}^3(\B)^N)$ (the crucial point is to get Lemma 11.2 with $\mathcal{C}^{1,\epsilon}$ instead of $\mathcal{C}^{\epsilon}$ which requires to increase the regularity of the equations).

\begin{theo}{\bf (\cite{Glob}, Theorem 7.1)}\label{theo:Globev} 
Assume that the previous conditions hold. For every $\tilde{\rho}=(\tilde{\rho}_1,\hdots,\tilde{\rho}_N)\in\mathcal{C}^{1,\epsilon}(\partial\Delta,\mathcal{C}^3(\B)^N)$ in a neighborhood of $\tilde{r}=(\tilde{r}_1,\hdots,\tilde{r}_N)$, we set for all $\zeta\in\partial\Delta$
$$\mathscr{E}_{\tilde{\rho}}(\zeta):=\{\omega\in\B | \tilde{\rho}_j(\zeta)(\omega)=0,\ 1\le j\le N\}.$$
Assume that the partial indices of $\mathscr{E}=\mathscr{E}_{\tilde{r}}$ along $f$ are non-negative, and denote by $\kappa$ the Maslov index of $\mathscr{E}$ along $f$. Then, there exist some open neighborhoods $V$ of $\tilde{r}$ in $\mathcal{C}^{1,\epsilon}(\partial\Delta,\mathcal{C}^3(\B)^N)$, $U$ of the origin in 
$\R^{\kappa+N}$, $W$ of $f$ in $\mathcal{C}^{1,\epsilon}(\partial\Delta,\B)$, and a map $\tilde{\mathcal{F}}:V\times U\to\mathcal{C}^{1,\epsilon}(\partial\Delta,\B)$ of class $\mathcal{C}^1$
such that
\begin{enumerate}[i)]
\item $\tilde{\mathcal{F}}(\tilde{r},0)=f$,
\item for all $(\tilde{\rho},t)\in V\times U$, the map $\zeta\mapsto\mathcal{F}(\tilde{\rho},t)(\zeta)-f(\zeta)$ is the boundary of a holomorphic disc attached to the totally real fibration $\mathscr{E}_{\tilde{\rho}}$,
\item  
$\forall\tilde{\rho}\in V$,
the map $\tilde{\mathcal{F}}(\tilde{\rho},.)$ is one-to-one,
\item if $g\in W$ satisfies $g(\zeta)\in\mathscr{E}_{\tilde{\rho}}(\zeta)$ on $\partial\Delta$ for some $\tilde{\rho}\in V$ and is such that $g-f$ extends holomorphically to $\Delta$, then there exists $t\in U$ such that $g=\tilde{\mathcal{F}}(\tilde{\rho},t)$.
\end{enumerate}
\end{theo}

Notice that if $f$ is the boundary map of a holomorphic disc, then this theorem describes all nearby discs attached to $\mathscr{E}_{\tilde{\rho}}$ for some $\tilde{\rho}$ close to $\tilde{r}$.

\subsection{Discs glued to a small perturbation of $\mathscr{N}Q$ \label{eqconormal}}
Let $Q$ be the hyperquadric in $\C^{n+1}$ defined  by 
$$r(z)=\Re e z_0-\tsp \overline{z}_\alpha Az_\alpha$$
for some invertible Hermitian matrix $A$, and fix $\bm{h}\in\mathscr{S}^*(Q)$.
For all $\zeta \in \partial \Delta$ we have $\bm{h}(\zeta)\in\mathscr{N}Q(\zeta)$.
Hence $\bm{h}_{|\partial\Delta}$ is the boundary map of a disc attached to the totally real fibration $\mathscr{N}Q$.
 
Moreover, $(z,w) \in \mathscr{N}Q(\zeta)$ if and only if $r(z)=0$ and $\zeta^{-1}w \in {\rm span}_{\R}\{\partial r_z\}$, which gives the equation (\ref{equationConormal}). 
Separating real and imaginary parts, we obtain $2n+2$ equations for $\mathscr{N}Q(\zeta)$:
  
\begin{equation}\label{eqr}
\left\{
\begin{array}{lll} 

\tilde{r}_0(\zeta)(z,w) & = & \frac{z_0+\overline{z}_0}{2} -\tsp\bar{z}_\alpha Az_{\alpha}  = 0,\\
\\
\tilde{r}_1(\zeta)(z,w) & = & i\frac{w_o}{\zeta}-i\zeta\overline{w}_0 = 0,\\
\\
\tilde{r}_2(\zeta)(z,w) & = & \left(w_1-2w_0\partial_{z_1}r(z)\right) + 
\left(\overline{w_1-2w_0\partial_{z_1}r(z)}\right) = 0,\\
\vdots& \vdots&\vdots\\
\tilde{r}_{n+1}(\zeta)(z,w) & = & \left(w_n-2w_0\partial_{z_n}r(z)\right) + 
\left(\overline{w_n-2w_0\partial_{z_n}r(z)}\right) = 0,\\

\tilde{r}_{n+2}(\zeta)(z,w) & = & i\left(w_1-2w_0\partial_{z_1}r(z)\right) - 
i\left(\overline{w_1-2w_0\partial_{z_1}r(z)}\right) = 0,\\
\vdots& \vdots&\vdots\\
\tilde{r}_{2n+1}(\zeta)(z,w) & = & i\left(w_n-2w_0\partial_{z_n}r(z)\right) - 
i\left(\overline{w_n-2w_0\partial_{z_n}r(z)}\right) = 0,\\
\end{array}
\right.
\end{equation}
where actually only $\tilde{r}_1$ depends on $\zeta$. The $(2(n+1)\times 2(n+1))$ matrix $G(\zeta)$ has the following expression

$$\left(\begin{matrix}

1/2  &   -L_1z_\alpha  &  \hdots  &  -L_nz_\alpha  & 0 &  0  &  \hdots  &   0\\

0  &   0  &  \hdots  &  0  & -i\zeta &  0  &  \hdots  &   0\\

0  &  2w_0a_{1,1}  &  \hdots  &  2w_0a_{n,1}  &  
2L_1z_\alpha  &  1  &     &   \\

\vdots  &  \vdots  &     &  \vdots  &  \vdots  &     &  \ddots  &   \\

0  &  2w_0a_{1,n}  &  \hdots  &  2w_0a_{n,n}  &  
2L_nz_\alpha  &     &     &  1\\

0  &  2iw_0a_{1,1}  &  \hdots  &  2iw_0a_{n,1}  &  
-2iL_1z_\alpha  &  -i  &     &    \\

\vdots  &  \vdots  &     &  \vdots  &  \vdots  &     &  \ddots  &   \\

0  &    2iw_0a_{1,n}  &  \hdots  &  2iw_0a_{n,n}  &  -2iL_nz_\alpha  &     &     &  -i\end{matrix}\right)
$$
where $L_j$ denotes the $j^{th}$ row of the matrix $A=(a_{i,j})_{1\le i,j\le n}$. 
 
Right multiplication by the constant matrix $\left(\begin{array}{ccc}2&0&0 \\0&\tsp A^{-1}&0\\ 0&0&I_{n+1}  \end{array}\right)$ does not change the partial indices, and gives us the matrix 

$$\left(\begin{matrix}

1  &   -z_1   &  \hdots  &  -z_n  & 0 &  0  &  \hdots  &   0\\

0  &   0  &  \hdots  &  0  & -i\zeta &  0  &  \hdots  &   0\\

0  &  2w_0  &     &    &  
2L_1z_\alpha  &  1  &     &   \\

\vdots  &&  \ddots  &     &  \vdots  &     &  \ddots  &   \\

0  &       &     &  2w_0&  
2L_nz_\alpha  &     &     &  1\\

0  &  2iw_0  &     &    &  
-2iL_1z_\alpha  &  -i  &     &    \\

\vdots  &&  \ddots  &     &  \vdots  &       &  \ddots  &   \\

0  &       &     &  2iw_0&  -2iL_nz_\alpha  &     &     &  -i\end{matrix}\right)
$$ 
Permuting the rows leads to 
$$\left(\begin{array}{cccccccccccc}

1 & -z_1 & -z_2 & \hdots & -z_{n-1} & -z_n & 0 & 0 & 0 & \hdots & 0 & 0\\
0 & 2w_0 & 0 & \hdots & 0 & 0 & 2L_1z_\alpha & 1 & 0 & \hdots & 0 & 0 \\
0 & 2iw_0 & 0 & \hdots & 0 & 0 & -2iL_1z_\alpha & -i & 0 & \hdots & 0 & 0 \\

\vdots & & & &  &  &\vdots & & \\

0 & 0 & 0 & \hdots & 0 &  2w_0 & 2L_nz_\alpha & 0 & 0 & \hdots & 0  & 1\\

0 &  0 &  0 & \hdots & 0 & 2iw_0 &  -2iL_nz_\alpha & 0 & 0 & \hdots & 0  & -i\\

0 & 0 & 0 & \hdots & 0 & 0 & -i\zeta & 0 & 0 & \hdots & 0 & 0\\
\end{array}\right)$$ 
and by permuting the columns, we get a triangular by block matrix

$$\left(\begin{array}{cccccccccccc}

1 & -z_1 & 0 & \hdots & 0 & -z_n & 0  & 0\\
0 & 2w_0 & 1 & \hdots & 0 & 0    & 0  & 2L_1z_\alpha \\
0 & 2iw_0 & -i& \hdots & 0 & 0   & 0 & -2iL_1z_\alpha \\

\vdots & & & &  &  &\vdots & & \\

0 & 0 & 0 & \hdots & 0 &  2w_0 & 1 &  2L_nz_\alpha\\

0 &  0 &  0 & \hdots & 0 & 2iw_0 & -i  & -2iL_nz_\alpha\\

0 & 0 & 0 & \hdots & 0 & 0 & 0 & -i\zeta\\
\end{array}\right)$$
with $(z,w)=\bm{h}(\zeta)$, $\zeta\in\partial\Delta$. According to (\ref{lift}), setting $b=\frac{2}{|1-a|^2}$, we have 
$2w_0=b\zeta|1-a\zeta|^2$ with $\zeta\in\partial\Delta$. By multiplying the even columns, except the last one, by 
$\frac{1}{b(1-\bar a \bar \zeta)}$ and the odd columns, except the first one, by $1-\bar a \bar \zeta$, we do not 
change the partial indices and we obtain the following matrix: 
 
$$G_1(\zeta):=\left(\begin{array}{cccccccccccc}
1 & \\
 & P & & (*) \\
 & &  \ddots  \\
& (0) &  & P   \\
 &  & &&  -i\zeta\\
\end{array}\right)$$ 
where $P=\left(\begin{array}{cc}
\zeta(1-a\zeta) & 1-\bar a \bar \zeta\\ 
i\zeta(1-a\zeta) &  -i(1-\bar a \bar \zeta)\\
\end{array}\right)$. So $\overline{P^{-1}}=\left(\begin{array}{cc}
\frac{1}{2\bar \zeta(1-\bar a \bar \zeta)} & \frac{i}{2\bar \zeta(1-\bar a \bar \zeta)} \\ 
  \frac{1}{2(1-a\zeta)} &  \frac{-i}{2(1-a\zeta)} \\
\end{array}\right)$. It follows that we are reduced to compute the partial indices of the matrix

\begin{equation}\label{eqB}
B_1(\zeta):=-\overline{G_1(\zeta)^{-1}}G_1(\zeta)=-\left(\begin{array}{cccccccccccc}
1 & \\
 & \overline{P^{-1}}P & & (*) \\
 & &  \ddots  \\
& (0) &  & \overline{P^{-1}}P   \\
 &  & &&  -\zeta^2\\
\end{array}\right)=
-\left(\begin{array}{cccccccccccc}
1 & \\
 & R & & (*) \\
 & &  \ddots  \\
& (0) &  & R   \\
&  & &&  -\zeta^2\\
\end{array}\right)
\end{equation}
where  
$R=\left(\begin{array}{cc}
0 & \zeta\\ 
\zeta &  0\\
\end{array}\right)$.
We need the following factorization lemma:
\begin{lem}{\bf (\cite{Glob}, Lemma 5.1)}\ 
Let $A:\partial\Delta\to GL_{2n+2}(\C)$ of class $\mathcal{C}^\epsilon$ ($0<\epsilon<1$), and denote by
 $\kappa_1\ge\hdots\ge\kappa_{2n+2}$ the partial indices of the map 
$\zeta \mapsto A(\zeta)\overline{A(\zeta)^{-1}}$. Then there exists a map $\Theta:\bar{\Delta}\to GL_{2n+2}(\C)$ 
of class $\mathcal{C}^\epsilon$, holomorphic on $\Delta$, such that 
$$\forall\zeta\in\partial\Delta,\ \Theta(\zeta)A(\zeta)\overline{A(\zeta)^{-1}}=\left(\begin{array}{ccc}\zeta^{\kappa_1}& &(0) \\ &\ddots& \\ (0)& &\zeta^{\kappa_{2n+2}}\end{array}\right)\overline{\Theta(\zeta)} .$$
\end{lem}

 By applying this lemma to the matrix $A=i\overline{G_1(\zeta)^{-1}}$,  
 we obtain a continuous map $\Theta:\bar{\Delta}\to GL_{2n+2}(\C)$, 
 holomorphic on $\Delta$ such that 
 $$\forall\zeta\in\partial\Delta,\ \Theta(\zeta)B_1(\zeta)=\left(\begin{array}{ccc}\zeta^{\kappa_0}& &(0) \\ &\ddots& \\ (0)& &\zeta^{\kappa_{2n+2}}\end{array}\right)
\overline{\Theta(\zeta)}.$$ 
Denote by $l=(l_1,\hdots,l_{2n+2})$ the last row of the matrix $\Theta$. It follows that for all  $\zeta\in\partial\Delta$
\begin{equation}\label{eqrow}
l(\zeta)B_1(\zeta)=\zeta^{\kappa_{2n+2}}\overline{l(\zeta)}
\end{equation}

\begin{itemize}
\item If $l_1 \not\equiv 0$ then (\ref{eqrow}) gives $-l_1(\zeta)=\zeta^{\kappa_{2n+2}}\overline{l_1(\zeta)}$ 
and by holomorphy of $\Theta$
we get $\kappa_{2n+2}\geq 0$.

\item If $l_1 \equiv 0$ then (\ref{eqrow}) gives two equations $-\zeta l_3(\zeta)=\zeta^{\kappa_{2n+2}}\overline{l_2(\zeta)}$ and 
$-\zeta l_2(\zeta)=\zeta^{\kappa_{2n+2}}\overline{l_3(\zeta)}$. 
\begin{itemize}
\item  If $l_2 \not\equiv 0$ then $l_3 \not\equiv 0$ we obtain  $\kappa_{2n+2}\geq 1$ by holomorphy. 
\item If $l_2 \equiv 0$ then $l_3 \equiv 0$ then we obtain two new equations   
$-\zeta l_5(\zeta)=\zeta^{\kappa_{2n+2}}\overline{l_4(\zeta)}$ and 
$-\zeta l_4(\zeta)=\zeta^{\kappa_{2n+2}}\overline{l_5(\zeta)}$ 
 from  (\ref{eqrow}).  
\end{itemize}
\end{itemize}
Continuing this process we reach the first nonzero element of $l$, say $l_{2p}$ for $p\geq 3$.  If $2p<2n+2$ then 
 (\ref{eqrow}) gives  
$-\zeta l_{2p+1}(\zeta)=\zeta^{\kappa_{2n+2}}\overline{l_{2p}(\zeta)}$
 and 
$-\zeta l_{2p}(\zeta)=\zeta^{\kappa_{2n+2}}\overline{l_{2p+1}(\zeta)}$ 
which imply that $\kappa_{2n+2}\geq 1$. If $2p=2n+2$ then 
 (\ref{eqrow}) gives the equation  $-\zeta^{2} l_{2n+2}(\zeta)=\zeta^{\kappa_{2n+2}}\overline{l_{2n+2}(\zeta)}$ implying 
  $\kappa_{2n+2}\geq 2$. Since $\kappa_1\ge\hdots\ge\kappa_{2n+2}$, we have proved:
\begin{lem}\label{lemparind}
The partial indices of $\mathscr{N}Q$ along $\bm{h}_{|\partial\Delta}$ are nonnegative, hence Theorem \ref{theo:Globev} applies to our situation.
\end{lem}
And more precisely, we have that the Maslov index of $\mathscr{N}Q$ along $\bm{h}_{|\partial\Delta}$ is $2n+2$.
This is a direct consequence of the following Lemma (see for instance \cite{BL} for a proof)
\begin{lem}\label{lemmasind}
Assume that the determinant $\mathrm{det} B$ is of class $\mathcal{C}^1$ on $\partial \Delta$. Then the Maslov index  of $B$ is given by
$$\mathrm{Ind}_{\mathrm{det}B(\partial \Delta)}(0)=\frac{1}{2\pi i}\int_{\partial \Delta}\frac{(\mathrm{det}B)'(\zeta)}{\mathrm{det}B(\zeta)}\,\mathrm{d}\zeta.$$
\end{lem}

Since $\bm{h}$ is smooth, Theorem \ref{theo:Globev} gives open neighborhoods $\tilde{V}$ of $\tilde{r}$ in $\mathcal{C}^{1,\epsilon}(\partial\Delta,\mathcal{C}^3(\B)^{2n+2})$, $U$ of the origin in $\R^{4n+4}$, $\tilde{W}$ of $\bm{h}_{|\partial\Delta}$ in $\mathcal{C}^{1,\epsilon}(\partial\Delta,T^*\C^{n+1})$, and a map $\tilde{\mathcal{F}}:\tilde{V}\times U\to\mathcal{C}^{1,\epsilon}(\partial\Delta,T^*\C^{n+1})$ of class $\mathcal{C}^1$ such that
\begin{itemize}
\item $\tilde{\mathcal{F}}(\tilde{r},0)=\bm{h}_{|\partial\Delta}$,
\item for all $(\tilde{\rho},t)\in \tilde{V}\times U$, the map $\zeta\mapsto\tilde{\mathcal{F}}(\tilde{\rho},t)(\zeta)$ is the boundary of a holomorphic disc attached to 
$$\mathscr{E}_{\tilde{\rho}}=\left\{\mathscr{E}_{\tilde{\rho}}(\zeta):=\{\omega\in\B | \tilde{\rho}_j(\zeta)(\omega)=0,\ 1\le j\le N\}\right\},$$
\item for every $\tilde{\rho}\in\tilde{V}$, the map $\tilde{\mathcal{F}}(\tilde{\rho},.)$ is one-to-one,
\item if $\bm{f}\in \tilde{W}$ is the boundary of a holomorphic disc attached to $\mathscr{E}_{\tilde{\rho}}$, then there exists $t\in U$ such that 
$\bm{f}=\tilde{\mathcal{F}}(\tilde{\rho},t)$.
\end{itemize}

Let $\B\subset\C^{n+1}$ be an open ball  centered at the origin. 
If the hypersurface $\Gamma^\rho$ is given by a defining function $\rho$ in a neighborhood of $r$ for the $\mathcal{C}^4(\B)$-topology, then the equation $\tilde{\rho}$ of the fibration $\mathscr{N}\Gamma^\rho$ is in a neighborhood of the equation $\tilde{r}$ of $\mathscr{N}Q$ for the 
$\mathcal{C}^{1,\epsilon}(\partial\Delta,\mathcal{C}^3(\B\times\B)^{2n+2})$-topology.
Thus we get:

\begin{theo}\label{descriptionTousDisques}
Let $Q=\{r=0\}$ where $r(z)=\Re e z_0-\tsp\bar{z}_\alpha A z_\alpha$ and $A$ is an invertible Hermitian $(n\times n)$ matrix. Fix $\bm{h}\in\mathscr{S}^*(Q)$ and an open ball $\B\subset\C^{n+1}$ such that $\bm{h}(\partial\Delta)\subset\B\times\B$. 
Then for any $0<\epsilon<1$, there exist some open 
neighborhoods $V$ of $r$ in $\mathcal{C}^4(\B)$ and $U$ of 0 in $\R^{4n+4}$, $\delta>0$, and a map
$\mathcal{F}:V \times U \to \mathrm{Hol}(\Delta,T^*\C^{n+1})\cap 
\mathcal{C}^{1,\epsilon}(\bar{\Delta},T^*\C^{n+1})$ of class $\mathcal{C}^1$ with respect to the 
$\mathcal{C}^{1,\epsilon}(\partial\Delta)$-topology, such that:
\begin{enumerate}[i)]
\item $\mathcal{F}(r,0)=\bm{h}_{|\partial\Delta}$,
\item for all $\rho\in V$, the map $\mathcal{F}(\rho,\cdot):U\to\{\bm{f}\in\mathscr{S}(\Gamma^\rho)\ |\ \|\bm{f}-\bm{h}\|_{\mathcal{C}^{1,\epsilon}(\partial\Delta)}<\delta\}$ is one-to-one and onto.
\end{enumerate}
\end{theo}

For a fixed $\bm{h}\in\mathscr{S}^*(Q)$, Theorem \ref{descriptionTousDisques} describes all nearby discs in $\mathscr{S}{\Gamma^\rho}$ as soon as $\rho$ is close to $r$ (note that the neigborhoods depends on the choice of $\bm{h}$, and of course of $Q$). Since the map $\mathcal{F}$ is of class $\mathcal{C}^1$, most properties of the discs in $\mathscr{S}(Q)$ remain true for the discs in $\mathscr{S}(\Gamma^\rho)$. Let us state more precise results.

\subsection{Discs tied to the origin}\label{subsubs:Perturbation}
The following statement is the analogue of Proposition \ref{2jet-discs}:
\begin{theo}\label{paramperturb}
Let $Q=\{r=0\}$ where $r(z)=\Re e z_0-\tsp\bar{z}_\alpha A z_\alpha$ and $A$ is an invertible Hermitian $(n\times n)$ matrix. Fix $\bm{h}\in\mathscr{S}^*(Q)$, and an open ball $\B\subset\C^{n+1}$ such that $\bm{h}(\partial{\Delta})\subset\B\times\B$. 
Then for any $0<\epsilon<1$, there exist 
$\varepsilon>0$ and $\delta>0$, both depending on $Q$ and $\bm{h}$, such that if $\|\rho-r\|_{\mathcal{C}^4(\B)}<\varepsilon$ (with $\rho$ in normal form), the set
$$\mathscr{S}_{\bm{h},\delta}^*(\Gamma^{{\rho}}):=\{\bm{f}=(f,g) \in\mathscr{S}^*(\Gamma^{{\rho}})\ |
\|\bm{f}-\bm{h}\|_{\mathcal{C}^{1,\epsilon}(\partial\Delta)}<\delta\}$$
forms a (2n+2)-real parameter family. Moreover, if $\tsp\overline{h_\alpha(0)}A h_\alpha(0)\not=0$, one can reduce the neighborhoods in order to get:
\begin{enumerate}[i)]
\item the discs $\bm{f}\in\mathscr{S}^*_{\bm{h},\delta}(\Gamma^\rho)$ are in $\mathcal{C}^{2,\epsilon}$ and satisfy $\bm{f}(\overline{\Delta})\subset\B\times\B$;
\item the map $\bm{f}\mapsto f(0)$ is a diffeomorphism of class $\mathcal{C}^1$ from $\mathscr{S}^*_{\bm{h},\delta}(\Gamma^\rho)$ onto its image;
\item the map $\bm{f}\mapsto (f_\alpha'(1),f'_0(1)(g_0)'(1))$ is a diffeomorphism of class $\mathcal{C}^1$ from $\mathscr{S}^*_{\bm{h},\delta}(\Gamma^\rho)$ onto its image.
\end{enumerate}
\end{theo}
\noindent Item $iii)$ implies that the map $\bm{f}\mapsto\bm{f}'(1)$ defined on $\mathscr{S}^*_{\bm{h},\delta}(\Gamma^\rho)$ is one-to-one.

\begin{proof}
Let $\mathcal{F}:V \times U \to 
\mathrm{Hol}(\Delta,T^*\C^{n+1})\cap \mathcal{C}^{1,\epsilon}(\bar{\Delta},T^*\C^{n+1})$ 
be the $\mathcal{C}^1$-map given by Theorem \ref{descriptionTousDisques}, and $\phi$ be the $\mathcal{C}^1$-map from $V \times U$ to $T^*\C^{n+1}$ defined by 
$\phi(\rho,t):=\mathcal{F}(\rho,t)(1)$. The map $(\rho,t)\mapsto \frac{\partial} {\partial t}\phi(\rho,t)$ is continuous from $V\times U$ 
to the Banach space $\mathcal{L}_c(T^*\C^{n+1})$ of continuous linear maps. According to Corollary \ref{submersion}, 
$\frac{\partial} {\partial t}\phi(r,0)$ is of rank $2n+2$. For $\rho$ in a neighborhood of $r$ and $t$ sufficiently small, $\frac{\partial} {\partial t}\phi(\rho,t)$ is thus of rank at least $2n+2$. This proves that for a fixed $\rho$, $\phi(\rho,\cdot):U\to N^*(\Gamma^\rho)$ is a submersion and hence $\phi(\rho,\cdot)^{-1}\{(0,\dots,0,1,0,\dots,0)\}$ is a submanifold of codimension $2n+2$ of the open set $U$. This exactly means that $\mathscr{S}_{\bm{h},\delta}^*(\Gamma^{{\rho}})$
is a $(4n+4)-(2n+2)$-parameter family which proves the first part of the theorem.
Statement $i)$ comes from the regularity results mentioned in section \ref{subsection-stationary}, since $N^*\Gamma^\rho$ is of class $\mathcal{C}^{1,\epsilon}$.

Let $\psi$ be the $\mathcal{C}^1$ map from $V \times U$ to $\C^{n+1}$ defined by 
$\psi(\rho,t):=\pi \circ \mathcal{F}(\rho,t)(0)$, where $\pi$ is the canonical projection onto the first 
$n+1$ components.    
According to Proposition \ref{2jet-discs}, $\frac{\partial} {\partial t}\psi(r,0)$ is an isomorphism since we have assumed $\tsp\overline{h_\alpha(0)}A h_\alpha(0)\not=0$. 
Hence  for $\rho$ sufficiently close to $r$ with respect to the $\mathcal{C}^4$-topology, we still get that $\frac{\partial} {\partial t}\psi(\rho,0)$ is of maximal rank and obtain the result by the inverse function theorem.

For $iii)$, we define the corresponding
$\varphi(\rho,t)$. This map is of class $\mathcal{C}^1$ since the linear map $\bm{f}\mapsto\bm{f}'(1)$ is smooth on 
$\mathrm{Hol}(\Delta,T^*\C^{n+1})\cap \mathcal{C}^{1,\epsilon}(\bar{\Delta},T^*\C^{n+1})$ for the
$\mathcal{C}^{1,\epsilon}(\partial\Delta)$-topology. According to Proposition \ref{2jet-discs}, 
$\frac{\partial} {\partial t}\varphi(r,0)$ is invertible, so we conclude by arguments similar to those of $ii)$.
\end{proof}

The proof of $ii)$ by means of the implicit function theorem also gives:

\begin{cor}\label{propopbis}
In Theorem \ref{paramperturb}, let $\Omega:=\{v=(v_0,v_\alpha)\in\C^{n+1}\ |\ \tsp\overline{v}_\alpha A v_\alpha\not=0\}$, and assume $h(0)\in\Omega$. Then we can restrict $\varepsilon$ and $\delta$ such that there exists an open neighborhood $O$ of $h(0)$ satisfying $O\subset\Omega$ and 
$O\subset\{f(0)\ |\ \bm{f}\in\mathscr{S}_{\bm{h},\delta/2}^*(\Gamma^{{\rho}})\}$ 
as soon as $\|\rho-r\|_{\mathcal{C}^4(\B)}<\varepsilon$.
\end{cor}

\section{2-jet determination of biholomorphic mappings\label{preuve}}

Let us notice first that if $G$ and $H$ are two germs of biholomorphisms mapping $\Gamma$ to $\Gamma'$ with the same 2-jet at point $p$, then $F=H^{-1}\circ G$ is a germ of biholomorphism equal to the identity up to order 2 and $F(\Gamma)=\Gamma$. Hence it is sufficient to prove the theorem with the hypothesis $\Gamma=\Gamma'$ and $F$ is equal to the identity up to order 2. 

Assume also that $p=0$ and that $\Gamma$ is given by the defining function $\rho$ in normal form (\ref{normalform}):
$$\rho(z)=x_0-\tsp \overline{z}_\alpha Az_\alpha+{b_0y_0^2+\sum_{j=1}^n(b_jz_j+\bar{b}_j\bar{z}_j)y_0+\phi(y_0,z_\alpha)}
\qquad\mathrm{with}\quad \phi(y_0,z_\alpha)=O\left(|(y_0,z_\alpha)|^3\right).$$

Since $F(0)=0$ and $dF_0=\mathrm{id}$, $\mathscr{S}^*(\Gamma)$ is invariant by $F$.
Theorem \ref{paramperturb} gives hence a globally invariant family, with a finite number of parameters, associated to $\Gamma$ with the picked point 0. We use this fact to get Theorem \ref{thmjet}.

Consider the non-degenerate hyperquadric $Q^A=\{z\ |\ r(z)=0\}$ where
$r(z)=\Re e z_0-\tsp\bar{z}_\alpha Az_\alpha$. By means of the scaling method (originally introduced in \cite{Pi89}) we reduce the situation to the one of small perturbations of $Q^A$. 

\subsection{Dilations}

Let us recall that $z_0$ is given a weight 2 and $z_\alpha$ a weight 1: 
for $t>0$, we consider the biholomorphism $\Lambda_t:(z_0,z_\alpha)\mapsto (t^2z_0,tz_\alpha)$.
This inhomogeneous dilation leaves $Q^A$ invariant. 

In the following, we make the restrictions $t\in (0;1]$ and $z\in\B$, for some open ball $\B\subset \C^{n+1}$ centered at 
the origin. 

\subsubsection{Dilation of the hypersurface}

Since $\phi$ is of class $\mathcal{C}^4$, we can decompose it as the sum of a polynomial $P(y_0,z_\alpha)$ of degree 3 and 
some function $\tilde{\phi}=O\left(|(y_0,z_\alpha)|^4\right)$. There exists a constant $C>0$ such that for all $z\in\B$, 
$\forall k=0,\hdots,4$,
$$\|d^k\tilde{\phi}_{(y_0,z_\alpha)}\|\le C\|(y_0,z_\alpha)\|^{4-k}.$$

Set $\Gamma_t:=\Lambda_t^{-1}(\Gamma)$ and $\rho_t:=\frac{1}{t^2}\rho\circ\Lambda_t$. We have
$$\rho_t(z)=x_0-\tsp \bar{z}_\alpha
Az_\alpha+b_0t^2y_0^2+\sum_{j=1}^n(b_jz_j+\bar{b}_j\bar{z}_j)ty_0+\frac{1}{t^2}P(t^2y_0,tz_\alpha)+\frac{1}{t^2}\tilde{\phi}(t^2y_0,tz_\alpha).$$
There exists a constant $C'$ such that $\|b_0t^2y_0^2+\sum_{j=1}^n(b_jz_j+\bar{b}_j\bar{z}_j)ty_0+\frac{1}{t^2}P(t^2y_0,tz_\alpha)\|_{\mathcal{C}^4(\B)}\le C't$.

Moreover, since $\|\tilde{\phi}(y_0,z_\alpha)\|\le C\|(y_0,z_\alpha)\|^4$, one can check that $\left\|\frac{1}{t^2}\tilde{\phi}(t^2y_0,tz_\alpha)\right\|_{\mathcal{C}^4(\B)}\le Ct$. Hence $\Gamma_t=\Gamma^{\rho_t}$ where the defining function $\rho_t$ satisfies $\|\rho_t-r\|_{\mathcal{C}^4(\B)}\le (C+C')t$. The constant $(C+C')$ only depends on $\rho$ and $\B$. In particular:

\begin{lem}\label{cst-t}
Let $V$ be a neighborhood of $r$ in $\mathcal{C}^4(\B)$. Then there exists $t_0>0$  such that $\rho_t \in V$ for all $0<t\leq t_0$.
\end{lem}

This means that for $t$ sufficiently small, the hypersurface $\Gamma_t$ is simultaneously in the equivalence class of $\Gamma$ and in a 
neighborhood of $Q^A$.

\subsubsection{Dilation of the germ of biholomorphism}
Since $F$ is the identity up to order 2, we get
$F(z)=z+\psi(z)$ and a constant $C_1>0$ such that for all $z\in\B$,
$$\|\psi(z)\|\le C_1\|z\|^3,\quad\|d\psi_z\|\le C_1\|z\|^2,\quad\|d^2\psi_z\|\le C_1\|z\|,\quad\|d^3\psi_z\|\le C_1$$
and similarly $F^{-1}(z)=z+\vartheta(z)$ gives a positive constant still denoted by $C_1$ such that for all $z\in F(\B)$,
$$\|\vartheta(z)\|\le C_1\|z\|^3,\quad\|d\vartheta_z\|\le C_1\|z\|^2,\quad\|d^2\vartheta_z\|\le C_1\|z\|,\quad\|d^3\vartheta_z\|\le C_1 .$$
The constant $C_1$ only depends on $\B$ and $F$.

\medskip

Set $F_t:=\Lambda_t^{-1}\circ F\circ\Lambda_t$.

\begin{lem}\label{cst-K}
There exists a constant $K$, depending only on $\B$ and $F$, such that for all $\bm{f}=(f,g) \in \mathcal{C}^{1,\epsilon}(\bar{\Delta},T^*\C^{n+1})$ with $\bm{f}(\overline{\Delta})\subset\B\times\B$, 
$$\forall t\in(0,1],\ \|{F_t}_*\bm{f}-\bm{f}\|_{\mathcal{C}^{1,\epsilon}(\partial\Delta)}\le t K\|\bm{f}\|_{\mathcal{C}^{1,\epsilon}(\partial\Delta)}^3 .$$
\end{lem}

\begin{proof}
Let $t\in(0,1]$ and $\bm{f}=(f,g)$ with $\bm{f}(\overline{\Delta})\subset\B\times\B$. 
We set $M:=\|\bm{f}\|_{\mathcal{C}^{1,\epsilon}(\partial\Delta)}$. By definition 
$${F_t}_*\bm{f}:\zeta\mapsto\left(F_t(f(\zeta)),g(\zeta) (d{F_t}_{f(\zeta)})^{-1} \right).$$

\medskip

Set $B:=F_t-\mathrm{id}$. Since $F(z)=z+\psi(z)$ and $F_t=\Lambda_t^{-1}\circ F\circ\Lambda_t$, we have
$$F_t(z)=z+\Lambda_t^{-1}\circ\psi\circ\Lambda_t(z)\ \Longrightarrow\ \|B(z)\|\le t C_1\|z\|^3.$$
We also get 
$${dF_t}_z=\mathrm{id}+\Lambda_t^{-1}\circ d\psi_{\Lambda_t(z)}\circ \Lambda_t\ \Longrightarrow\ \|{dB}_z\|\le t C_1\|z\|^2$$
and
$${d^2F_t}_z(v,w)=\Lambda_t^{-1}\left(d^2F_{\Lambda_t(z)}(\Lambda_t(v),\Lambda_t(w))\right)\ \Longrightarrow\ \|{d^2B}_z\|\le t C_1\|z\|.$$
Thus 
$$\|B\circ f\|_\infty\le t C_1 \| f\|^3_\infty \le t C_1  M^3.$$
Since $(B\circ f)'(\zeta)=dB_{f(\zeta)}\cdot f'(\zeta)$ we have, 
$$\|(B\circ f)'\|_\infty\le t C_1\|f\|^2_\infty\|f'\|_\infty\le  t C_1 M^3.$$ 
Moreover, 
for every $\zeta\not=\eta\in\partial\Delta$, 
\begin{eqnarray*}
\|(B\circ f)'(\zeta)-(B\circ f)'(\eta)\|&\le&\|\left(dB_{f(\zeta)}-dB_{f(\eta)}\right)\cdot f'(\zeta)\|+\|dB_{f(\zeta)}\cdot\left(f'(\zeta)-f'(\eta)\right)\|\\
 &\le&\underset{z\in\bar{\B}(0,\|f\|_\infty)}{\mathrm{max}}\|d^2B_z\|\times\|f(\zeta)-f(\eta)\|\times \|f'\|_\infty+
 tC_1\|f\|^2_\infty\|f'(\zeta)-f'(\eta)\|\\
 &\le& 2tC_1 M^3|\zeta-\eta|^\epsilon.
\end{eqnarray*}
This implies that 
\begin{equation}\label{base}
\|F_t\circ f-f\|_{\mathcal{C}^{1,\epsilon}(\partial\Delta)}=\|B\circ f\|_{\mathcal{C}^{1,\epsilon}(\partial\Delta)}\le t K M^2
\end{equation}
for some positive constant $K>0$ depending only on $\B$ and $F$.

\bigskip

For the lift part, we have to consider
$$(d{F_t}_{z})^{-1}={d(F_t^{-1})}_{F_t(z)}=\Lambda_t^{-1}\circ d(F^{-1})_{\Lambda_t\circ F_t(z)}\circ\Lambda_t=\Lambda_t^{-1}\circ d(F^{-1})_{F(\Lambda_t(z))}\circ\Lambda_t.$$
For every $z\in\B$, we set $L(z):=(d{F_t}_{z})^{-1}-\mathrm{id}=\Lambda_t^{-1}\circ\left(d(F^{-1})_{F(\Lambda_t(z))}-\mathrm{id}\right)\circ\Lambda_t$ (the map $L$ takes its values in the set of linear maps), and $\|dF\|_\infty:=\underset{z\in\bar{\B}}{\mathrm{max}}\|dF_z\|$.
We have 
$$\|L(z)\|=\|\Lambda^{-1}_t\circ d\vartheta_{F(\Lambda_t(z))}\circ \Lambda_t\|\le\frac{1}{t^2}\|d\vartheta_{F(\Lambda_t(z))}\|t\le \frac{1}{t}C_1\|F(\Lambda_t(z))\|^2$$ 
with $\|F(\Lambda_t(z))\|\le \|dF\|_\infty\|\Lambda_t(z)\|\le \|dF\|_\infty t\|z\|$, so
\begin{equation}\label{majL}
\forall z\in\B,\ \|L(z)\|\le t C \|z\|^2.
\end{equation}
for some positive constant $C>0$ depending only on $\B$ and $F$.
Since 
$$dL_z(v)\cdot w=\Lambda_t^{-1}\circ d^2(F^{-1})_{F(\Lambda_t(z))}\left(dF_{\Lambda_t(z)}(\Lambda_tv),\Lambda_tw\right),$$
we also have
$$\|dL_z\|\le\frac{1}{t^2}\|d^2\vartheta_{F(\Lambda_t(z))}\|\times\|dF_{\Lambda_t(z)}\|t^2\le C_1\|F(\Lambda_t(z))\|\times\|dF\|_\infty$$
and thus
\begin{equation}\label{majdL}
\forall z\in\B,\ \|dL_z\|\le t C \|z\|.
\end{equation}
for some positive constant still denoted by $C$ depending only on $\B$ and $F$.
We also have to compute $d^2L$:
\begin{eqnarray*}
d^2L_z(h,v)\cdot w&=&\Lambda_t^{-1}\circ d^2(F^{-1})_{F(\Lambda_t(z))}\left(d^2F_{\Lambda_t(z)}(\Lambda_th,\Lambda_tv),\Lambda_tw\right)\\
& \ \ \ &+\Lambda_t^{-1}\circ d^3(F^{-1})_{F(\Lambda_t(z))}\left(dF_{\Lambda_t(z)}(\Lambda_th),dF_{\Lambda_t(z)}(\Lambda_tv),\Lambda_tw\right).
\end{eqnarray*}
Thus 
\begin{eqnarray*}
\|d^2L_z\|&\le& \frac{1}{t^2}\|d^2(F^{-1})_{F(\Lambda_t(z))}\|\times\|d^2F_{\Lambda_t(z)}\|t^3+\frac{1}{t^2}\|d^3(F^{-1})_{F(\Lambda_t(z))}\|\times\|dF_{\Lambda_t(z)}\|t\|dF_{\Lambda_t(z)}\|t^2\\
& \le &  tC_1\|F(\Lambda_t(z))\| tC_1\|z\|
+tC_1\|dF\|_\infty^2\\
 &\le& tC_1\|dF\|_\infty^2 t^2\|z\|^2+tC_1\|dF\|_\infty^2.
\end{eqnarray*}
So we get
\begin{equation}\label{majd2L}
\forall z\in\B,\ \|d^2L_z\|\le t C
\end{equation}
for some positive constant still denoted by $C$ depending only on $\B$ and $F$.
\medskip

Now we want to estimate $\|g\cdot (d{F_t}_{f})^{-1}-g\|_{\mathcal{C}^{1,\epsilon}(\partial\Delta)}=\|g\cdot L\circ f\|_{\mathcal{C}^{1,\epsilon}(\partial\Delta)}$. In view of (\ref{majL}) we get
$$\|g\cdot L\circ f\|_\infty\le M\|L\circ f\|_\infty\le t C M \|f\|_\infty^2\le t C M^3.$$
Finally, $(g\cdot L\circ f)'(\zeta)=g'(\zeta)\cdot L\circ f(\zeta) + g(\zeta)\cdot dL_{f(\zeta)}f'(\zeta)$ and 
\begin{eqnarray*}
\|(g\cdot L\circ f)'\|_\infty&\le & M\|L\circ f\|_\infty+M^2\|dL_f\|_\infty\\
 &\le& tCM^3+M^2tC\|f\|_\infty\\
&\le&2 t C M^3.
\end{eqnarray*}
For every $\zeta\not=\eta\in\partial\Delta$,
\begin{eqnarray*}
\|g'(\zeta)\cdot L\circ f(\zeta)-g'(\eta)\cdot L\circ f(\eta)\|&\le&\|\left(g'(\zeta)-g'(\eta)\right)\cdot L\circ f(\zeta)\|+\|g'(\eta)\cdot\left( L\circ f(\zeta)- L\circ f(\eta)\right)\|\\
&\le&M|\zeta-\eta|^\epsilon t C M^2 +M \underset{z\in\bar{\B}(0,\|f\|_\infty)}{\mathrm{max}}\|dL_z\|\times\|f(\zeta)-f(\eta)\|\\
&\le& tCM^3|\zeta-\eta|^\epsilon + tCM^3|\zeta-\eta|^\epsilon\\
&\le& 2tCM^3|\zeta-\eta|^\epsilon,
\end{eqnarray*}
and similarly
\begin{eqnarray*}
\|g(\zeta)\cdot dL_{f(\zeta)}f'(\zeta)-g(\eta)\cdot dL_{f(\eta)}f'(\eta)\|&\le& \|\left(g(\zeta)-g(\eta)\right)\cdot dL_{f(\zeta)}f'(\zeta)\|+\|g(\eta)\cdot\left(dL_{f(\zeta)}-dL_{f(\eta)}\right)f'(\zeta)\|\\
 & &\ \ +\|g(\eta)\cdot dL_{f(\eta)}\left(f'(\zeta)-f'(\eta)\right)\|\\
&\le& M|\zeta-\eta|^\epsilon tC\|f\|_\infty M+M\underset{z\in\bar{\B}(0,\|f\|_\infty)}{\mathrm{max}}\|d^2L_z\|M|\zeta-\eta|^\epsilon M\\
 & &\ \ +MtC\|f\|_\infty M|\zeta-\eta|^\epsilon\\
&\le& 3tCM^3|\zeta-\eta|^\epsilon
\end{eqnarray*}
by means of (\ref{majdL}) and (\ref{majd2L}). Hence
\begin{equation*}
\|(g\cdot L\circ f)'\|_{\mathcal{C}^{\epsilon}(\partial\Delta)}\le 5tCM^3.
\end{equation*}
This finally gives
\begin{equation}\label{norme1alpha}
\|g\cdot (d{F_t}_{f})^{-1}-g\|_{\mathcal{C}^{1,\epsilon}(\partial\Delta)}=\|g\cdot L\circ f\|_{\mathcal{C}^{1,\epsilon}(\partial\Delta)}\le t K M^3
\end{equation}
for some positive constant still denoted by  $K$ depending only on $\B$ and $F$.
We conclude by means of (\ref{base}) and (\ref{norme1alpha}).
\end{proof}

\subsection{Proof of Theorem \ref{thmjet}}
Fix some point $z\in\Omega=\{(\gamma\tsp\bar{v}Av,v)\ |\ v\in\C^n,\ \tsp\bar{v}Av\not=0,\ \Re
e(\gamma)>1\}$ and let  $\bm{h}=(h,\zeta h^*)\in\mathscr{S}^*Q$ be the unique holomorphic disc 
satisfying $h(0)=z$. Let $\B \subset \C^{n+1}$ be an open ball centered at the origin,  such that 
$\bm{h}(\overline{\Delta})\subset\B\times\B$. 

By Theorem \ref{paramperturb}, there exist $\varepsilon>0$ and $\delta>0$ (depending only on $r$ and $\bm{h}$) 
such that if
$\|\tilde{\rho}-r\|_{\mathcal{C}^4(\B)}<\varepsilon$, then the set 
$$\mathscr{S}_{\bm{h},\delta}^*(\Gamma^{\tilde{\rho}})=\{\bm{f}\in\mathscr{S}^*(\Gamma^{\tilde{\rho}})\ |
\|\bm{f}-\bm{h}\|_{\mathcal{C}^{1,\epsilon}(\partial\Delta)}<\delta\}$$
is a $(2n+2)$-parameter family parametrized by $\bm{f}\mapsto f(0)$, and the map
$$\begin{array}{rcl}
\mathscr{S}_{\bm{h},\delta}^*(\Gamma^{\tilde{\rho}})&\to&\C^{n+1}\times\C^{n+1}\\
\bm{f}&\mapsto&\bm{f}'(1)
\end{array}$$
is well-defined and one-to-one.
Let $O$ be an open neighborhood of $z$ given by Corollary \ref{propopbis}.

\bigskip

Let $K$ be the constant given by Lemma \ref{cst-K} and
 set $M:=\|\bm{h}\|_{\mathcal{C}^{1,\epsilon}(\partial\Delta)}$. In view of Lemma \ref{cst-t},
there exists $t_0>0$ such that for all 
$0<t\le t_0,\ \|\rho_t-r\|_{\mathcal{C}^4(\B)}<\varepsilon$, and we can ask
for $t_0<\frac{\delta}{2K(M+1)^3}$. Assume that $0<t\le t_0$ and let us prove that ${F_t}={\Lambda_t^{-1}\circ F\circ \Lambda_t}$ 
is equal to the identity on $O$. Since $O$ is an open set and $F_t$ is holomorphic, this will force $F_t$ to be the identity, and so $F=\mathrm{id}$.

Let $w\in O$. According to the choice of $O$, there exists a unique $\bm{f}=(f,\zeta f^*)\in\mathscr{S}_{\bm{h},\delta/2}^*(\Gamma^{{\rho_t}})$ such that  $f(0)=w$. Then ${F_t}_*\bm{f}\in\mathscr{S}^*(\Gamma^{{\rho_t}})$ since $F_t(0)=0$ and $({dF_t})_0=\mathrm{id}$. Moreover, 
according to Lemma \ref{cst-K} and to the choice of $t$
$$\|{F_t}_*\bm{f}-\bm{h}\|_{\mathcal{C}^{1,\epsilon}(\partial\Delta)}\le\|{F_t}_*\bm{f}-\bm{f}\|_{\mathcal{C}^{1,\epsilon}(\partial\Delta)}+\|\bm{f}-\bm{h}\|_{\mathcal{C}^{1,\epsilon}(\partial\Delta)}\le tK\|\bm{f}\|_{\mathcal{C}^{1,\epsilon}(\partial\Delta)}^3+\frac{\delta}{2}<\delta.$$ 
It follows that ${F_t}_*\bm{f}\in \mathscr{S}_{\bm{h},\delta}^*(\Gamma^{{\rho_t}})$. 
But then ${F_t}_*\bm{f}$ and $\bm{f}$  are both in $\mathscr{S}_{\bm{h},\delta}^*(\Gamma^{{\rho_t}})$. 
Since $F_t$ is the identity up to order 2, ${F_t}_*\bm{f}$ and $\bm{f}$ have the same derivative at 1, which gives 
${F_t}_*\bm{f}=\bm{f}$. In particular $F\circ f=f$ and thus $F_t(w)=F_t(f(0))=f(0)=w$. This concludes the proof.

\qed

\subsection{Proof of Theorem \ref{thmcr}}
Let $\Gamma,\ \Gamma'\subset \C^{n+1}$ be two strictly pseudoconvex hypersurfaces of class $\mathcal{C}^4$. 
Consider a germ at $p=0 \in \Gamma$ of a CR diffeomorphism $F$ of class  $\mathcal{C}^3$  satisfying $F(\Gamma)=\Gamma'$.
Let $\rho$ (resp. $\rho'$) be a $\mathcal{C}^4$ local defining function of $\Gamma$ (resp. $\Gamma'$). 
 According to Lewy's CR extension Theorem 
for Levi non-degenerate hypersurfaces \cite{Lew}, there is an open neighborhood $U$ of $p$ in $\C^{n+1}$ such that $F$ can be 
extended as a holomorphic mapping $\hat{F}$ on $\{\rho<0\}\cap U$ and continuous up to $\Gamma$. The extension $\hat{F}$  inherits 
the smoothness of it boundary values and thus $\hat{F}\in \mathcal{C}^{3}\left(\overline{\{\rho<0\}}\cap U\right)$ 
(see Theorem 7.5.1 in \cite{BER3}).

Moreover $\hat{F}$ is $\mathcal{C}^1$ up to the boundary $\Gamma$ and holomorphic on one side, so its 1-jet at $0$ only depends only on the 1-jet at 0 of $\hat{F}_{|\Gamma}=F$, and $d\hat{F}_0$ is invertible. Hence
 $\hat{F}$ is a biholomorphism from $\{\rho<0\}\cap U$ onto $\{\rho'<0\}\cap U'$ for some open neighborhood $U'$ of $F(p)$.
Since $d\hat{F}$ is $\mathcal{C}^1$ up to the boundary and holomorphic on one side, $d^2\hat{F}_0$ only depends on $d\hat{F}_0$ and $d^2F_0$. 
 It follows that if $\hat{G}$ is another biholomorphic extension of $F$, then $\hat{F}$ and $\hat{G}$ have the same 2-jet at 0. 
 As observed in the introduction, the proof of Theorem \ref{thmjet} carries over to the one sided situation in the strictly pseudoconvex case and 
therefore the rest of the proof of Theorem \ref{thmcr} follows the proof of Theorem \ref{thmjet}. 

\qed

\section{Almost complex case}

We refer to \cite{Sikorav} for  definitions and basic properties related to almost complex manifolds. 
In the following, we suppose that $\R^{2n+2}$ is endowed with an almost complex structure $J$ of class $\mathcal{C}^{3}$. In this section, we denote by $J_{st}$ the standard complex structure on $\R^{2n+2}$ and by $\mathcal{J}$ the set of almost complex structures on $\R^{2n+2}$ of class $\mathcal{C}^{3}$ equipped with the $\mathcal{C}^3$-topology.

Regarding the smoothness, let us recall the following statement, which will allow us to introduce the 2-jet of germs of pseudo-biholomorphisms:

\begin{prop}\label{reg}\cite{Lee} Let $(M^{2n},J)$ and $(M'^{2m},J')$ be two almost complex manifolds endowed with $\mathcal{C}^r$ almost complex structures, $r>1,\ r\notin\N$. Then, pseudoholomorphic maps from $M$ to $M'$ are $\mathcal{C}^{r+1}$-smooth.
\end{prop}
\noindent Proposition \ref{reg} is given in \cite{Lee} in the $\mathcal{C}^\infty$-smooth case, but the proof goes through the H\"olderian setting. The statement originally concerned pseudoholomorphic discs (that is, $M$ is the standard unit disc) \cite{Sikorav}.

\subsection{Real submanifolds in almost complex manifolds}

Let $\rho$ be a smooth real valued function on $\left(\R^{2n+2},J\right).$
We denote by $d^c_J\rho$ the differential form defined by $d^c_J\rho\left(v\right):=-d\rho\left(Jv\right)$. 
The {\it Levi form} of $\rho$ at a point $p\in \R^{2n+2}$ and a vector 
$v \in T_p \R^{2n+2}$ is defined by $\mathcal{L}_J\rho\left(p,v\right):=dd^c_J\rho(p)\left(v,J(p)v\right).$ 

\begin{defi} 
A hypersurface $\Gamma=\{\rho=0\}$ is {\it $J$-Levi non-degenerate} at a point $p\in\Gamma$ if the restriction to 
$T_p^J\Gamma:=T_p\Gamma \cap J(p)T_p\Gamma$ of the Levi form 
 $\mathcal{L}_J\rho\left(p,v\right)$ is non-degenerate.
\end{defi}

We need to define an almost structure on the cotangent bundle. The following construction is due to   I. Sato \cite{Sato}.
We denote by $(x_{1},\cdots,x_{2n+2})$ the canonical coordinates on $\R^{2n+2}$ and by 
$(x_{1},\cdots,x_{2n+2},q_{1},\cdots,q_{2n+2})$ the canonical coordinates on the cotangent bundle $T^{*}\R^{2n+2}$.
We consider the almost complex structure defined on the cotangent bundle by: 
\begin{equation}\label{defreleve}
\bm{J}=\left(\begin{array}{ccccc} 
J_{j}^{i} & 0 \\ M^{i}_{j}& J_{i}^{j}\\
\end{array}\right)\mbox{ with } M^{i}_{j} =\sum_{k=1}^{2n+2}
\frac{q_{k}}{2}\left(\frac{\partial J^{k}_{i}}{\partial x_{j}} -
\frac{\partial J^{k}_{j}}{\partial x_{i}} +
J^{k}_{s}J^{q}_{i}\frac{\partial J^{s}_{j}}{\partial x_{q}} -
J^{k}_{s}J^{q}_{j}\frac{\partial J^{s}_{i}}{\partial x_{q}}\right).
\end{equation}
By construction $\bm{J}$ is an almost complex structure and has the following invariance property: 
if $F$ is a biholomorphism between $(\R^{2n+2},J)$ and $(\R^{2n+2},J')$, then its lift to the cotangent bundle 
is a biholomorphism between $(T^*(\R^{2n+2}),\bm{J})$ and $(T^*\R^{2n+2},\bm{J'})$.
Also, notice that since $J$ is of class $\mathcal{C}^{3}$ then  $\bm{J}$ is of class $\mathcal{C}^{2}$.

\begin{defi}
The {\it conormal bundle} $N^*_J\Gamma$ of a real hypersurface $\Gamma \subset \R^{2n+2}$ is the real subbundle of 
$T^*_{(1,0)}M$ defined by 
$$N^*_J\Gamma:= \{ \phi \in T^*_{(1,0)}\R^{2n+2} | \Re e \phi_{|T\Gamma} = 0\}.$$  
\end{defi}
The conormal bundle $N^*_J\Gamma$ of $\Gamma$  can be identified  with any of the
following subbundles $\{\phi \in T^*\R^{2n+2} | \phi_{|T\Gamma}=0\}$ and 
$\{\phi \in T^*\R^{2n+2} | \phi_{|JT\Gamma}=0\}$. The following result is the generalization of Proposition \ref{propco} in the almost complex setting:
\begin{prop}\label{propco2}{\cite{GS}}
A real hypersurface $\Gamma \subset \R^{2n+2}$ is Levi non-degenerate if and only if its 
conormal bundle $N^*_J\Gamma$ (out of the zero section ) is a totally real submanifold of $(T^*\R^{2n+2},{\bm J})$
of dimension $2n+2$, \textsl{i.e.} $TN^*_J\Gamma \cap {\bm J}TN^*_J\Gamma=\{0\}$.
\end{prop}

\subsection{Pseudoholomorphic discs attached to a deformation of a non-degenerate
hyperquadric}

We follow here the approach used in \cite{CGS} for smooth deformations of
the unit sphere.
Let us recall that for a structure $J$ sufficiently close to the standard
complex structure, the
$J$-holomorphy equation for a disc $f: \Delta \rightarrow (\R^{2n+2},J)$
is given by
\begin{equation}\label{equa0}
\bar\partial_{J}f:= \frac{\partial f}{\partial \bar \zeta}
+\overline{Q_0}(J,f)\frac{\partial f}{\partial \zeta} = 0
\end{equation}
where the conjugate linear operator
$\overline{Q_0}(J,f)=\left[(J+J_{st})^{-1}(J-J_{st})\right]\circ f$ satisfies
$\overline{Q_0}(J_{st},\cdot) \equiv 0$.

The $\bm J$-holomorphy equation for a disc
$\bm f=(f,g): \Delta \rightarrow (T^*\R^{2n+2},\bm J)$ is given by
\begin{equation}\label{equa1}
\bar\partial_{\bm J}\bm f:=\left(\bar\partial_{J}f,
\frac{\partial \bar g}{\partial \zeta} +\overline{Q_1}(J,f)\frac{\partial \bar g}{\partial
\bar \zeta}+Q_2(J,f)\bar g\right)=0
\end{equation}
where the operators $Q_1$ and $Q_2$ satisfy $Q_1(J_{st},.)=Q_2(J_{st},.)=0$. More precisely, we have
$\overline{Q_1}(J,f)=(\tsp J(f)+\tsp J_{st})^{-1}(\tsp J(f)-\tsp J_{st})$ and 
$Q_2(J,f)\bar g=(\tsp J(f)+\tsp J_{st})^{-1}M(f,\bar g)\frac{\partial f}{\partial x}$ where $M$ is defined in (\ref{defreleve}).

\begin{defi}
A $J$-holomorphic disc $f$ glued to a real hypersurface $\Gamma$ is $J$-{\it stationary} 
if there exists a $\bm{J}$-holomorphic lift $\bm{f}=(f,g)$ of $f$ to the cotangent bundle $T^*\R^{2n+2}$, continuous up to the boundary, such that $\forall\zeta \in\partial\Delta,\ \bm{f}(\zeta)\in\mathscr{N}_J\Gamma(\zeta)$
where
$$\mathscr{N}_J\Gamma(\zeta):=\{(z,\zeta w)\ |\ z\in\Gamma,\ w\in N_{z,J}^*\Gamma\setminus\{0\}\}.$$
The set of these lifted discs $\bm{f}=(f,g)$, with $f$ non-constant, is denoted by $\mathscr{S}_J(\Gamma)$.
\end{defi}
We also set
$$\mathscr{S}^*_J(\Gamma):=\{\bm{f}=(f,g)\in\mathscr{S}_J(\Gamma)\ |\ f(1)=0,\ g(1)=(1,0,\hdots,0)\}.$$
Regularity for pseudoholomorphic discs glued to a totally real submanifold has been studied by  Coupet-Gaussier-Sukhov (Proposition 4.7 in \cite{CGS2}, see also Theorem 1 in \cite{BL2}). Under our hypotheses ($\Gamma$ is a Levi non-degenerate hypersurface of class $\mathcal{C}^{4}$, $J$ is of class $\mathcal{C}^3$),
discs $\bm{f} \in \mathscr{S}_J(\Gamma)$ are in $\mathcal{C}^{1,\epsilon}(\bar \Delta)$ for any $\epsilon>0$.

Let $Q$ be the hyperquadric in $\C^{n+1}$ defined  by 
$$r(z)=\Re e z_0-\tsp \overline{z}_\alpha Az_\alpha$$
for some invertible Hermitian matrix $A$, and fix $\bm{h}=(h,g)\in\mathscr{S}^*(Q)$ such that $\bm{h}(\bar \Delta)\subset \B\times \B$ for some open ball 
$\B\subset \R^{2n+2}$ centered at the origin. 
We want to describe, for $J$ sufficiently close to the standard structure $J_{st}$, the family of $J$-stationary discs glued 
to a small perturbation of  the hyperquadric $Q$ and close enough to $h$.

Let $\mathscr{E}_{\tilde{\rho}}=\left\{\mathscr{E}_{\tilde{\rho}}(\zeta):=\{\omega\in\B | \tilde{\rho}_j(\zeta)(\omega)=0,\ 1\le j\le {2n+2}\}\right\}$ 
be a totally real deformation of the totally real fibration $\mathscr{N}Q$. 
A $\bm{J}$-holomorphic disc $\bm{f} \in \mathcal{C}^{1,\epsilon}(\bar{\Delta},T^*\C^{n+1})$ is attached
to $\mathscr{E}_{\tilde{\rho}}$  if and only if it satisfies:
$$
\left\{
\begin{array}{lll}
\tilde{\rho}(\zeta)(\bm{f}(\zeta))&=& 0, \ \ \ \zeta \in \partial \Delta\\
 & & \\
\bar{\partial}_{\bm{J}}\bm{f}(\zeta) &=& 0, \ \ \ \zeta \in \Delta.
\end{array}
\right.
$$
Let $\mathcal U$ be a neighborhood of $(\bm{f},\tilde{r},J_{st})$ in the 
space $\mathcal{C}^{1,\epsilon}(\bar{\Delta},T^*\C^{n+1}) \times \mathcal{C}^{1,\epsilon}(\partial\Delta,\mathcal{C}^3(\B\times \B)^{2n+2}) 
\times \mathcal{J}$ (where $\tilde{r}$ is the equation of $\mathscr{N}_{J_{st}}Q$ as in Section \ref{eqconormal}), and  define the map
\begin{equation}\label{map}
\begin{array}{llccccl}
\Phi &:& \mathcal U & \rightarrow & \mathcal{C}^{1,\epsilon}(\partial \Delta,\R^{2n+2}) &\times&
 \mathcal C^{\epsilon}(\Delta)\\
  & & (\bm{f},\tilde{\rho},J) & \mapsto & (\ v_{\bm{f},\tilde{\rho}}&,& 
\bar{\partial}_{\bm{J}}\bm{f}\ )
\end{array}
\end{equation}
where $v_{\bm{f},\tilde{\rho}}: \partial \Delta \rightarrow \R^{2n+2}$ is defined by $v_{\bm{f},\tilde{\rho}}(\zeta)=\tilde{\rho}(\zeta)(\bm{f}(\zeta))$. 
The first component of $\Phi$ is the map previously introduced for the standard case. The second component of $\Phi$ checks the pseudoholomorphy condition, and is of class $\mathcal{C}^1$ according to the explicit expression of $\bar{\partial}_{\bm{J}}\bm{f}$ and the following lemma:

\begin{lem}
Let $\psi:\B^n\to\C$ be of class $\mathcal{C}^2$ on an open ball $\B^n\subset\C^n$ and $\B^k$ be a ball in $\C^k$. Then the map
$$
\begin{array}{cccl}
\mathcal{C}^{\epsilon}(I,\B^k)\times\mathcal{C}^2(\B^k,\B^n)&\to&\mathcal{C}^\epsilon(I,\C)\\
(f,g)&\mapsto&\psi\circ g\circ f
\end{array}
$$
is of class $\mathcal{C}^1$ ($I$ is a bounded subset of the complex plane).
\end{lem}
Note that this is the reason why we need the almost complex structures to be $\mathcal{C}^3$.

\begin{proof}
Setting
 $A(g)=\psi\circ g$ and $B(f,h)=h\circ f$, we get $\psi\circ g\circ f=B(f,A(g))$. The map $A:\mathcal{C}^2(\B^k,\B^n)\to\mathcal{C}^2(\B^k,\C)$ is of class $\mathcal{C}^1$ (\cite{HT}, Lemma 5.1-a with $k=2$, $\alpha=0$, $s=1$). So it remains to prove that the map 
$B:\mathcal{C}^{\epsilon}(I,\B^k)\times\mathcal{C}^2(\B^k,\C)\to\mathcal{C}^\epsilon(I,\C)$ is of class $\mathcal{C}^1$.
This follows from Lemma 6.1 in \cite{Glob}, but since our case is simpler, we give the proof for completeness. 

For any fixed $f\in\mathcal{C}^\epsilon(I,\B^k)$, the map $g\mapsto g\circ f$ is linear continuous from $\mathcal{C}^2(\B^k,\C)$ to $\mathcal{C}^\epsilon(I,\C)$, hence of class $\mathcal{C}^1$ and $d_2B_{(f,g)}\cdot\texttt{g}:x\mapsto \texttt{g}\circ f(x)$. Since  
\begin{equation}\label{IT}
\texttt{g}\circ (f+\texttt{f})(x)-\texttt{g}\circ f(x)=\int_0^1d\texttt{g}_{f(x)+t\texttt{f}(x)}\cdot \texttt{f}(x)\,\mathrm{d}t ,
\end{equation}
we get $\|\texttt{g}\circ(f+\texttt{f})-\texttt{g}\circ f\|_\epsilon\lesssim\|\texttt{g}\|_{\mathcal{C}^2}\|\texttt{f}\|_\epsilon$. Thus the partial differential $d_2B$ is continuous at any point.

The first partial map $f\mapsto g\circ f$ is of class $\mathcal{C}^1$ from $\mathcal{C}^\epsilon(I,\B^k)$ to $\mathcal{C}^\epsilon(I,\C)$ for any fixed $g\in\mathcal{C}^2(\B^k,\C)$ (\cite{HT}, Lemma 5.1-a with $k=0$, $s=1$), and the corresponding partial differential is given by $d_1B_{(f,g)}\cdot \texttt{f}:x\mapsto dg_{f(x)}\cdot \texttt{f}(x)$ for any $\texttt{f}\in\mathcal{C}^\epsilon(\partial\Delta,\B^k)$. Moreover,
$$\|d_1B_{(f+\texttt{f},g+\texttt{g})}-d_1B_{(f,g)}\|\lesssim\|d(g+\texttt{g})_{f+\texttt{f}}-d{g}_{f}\|_\epsilon\le\|d{{g}}_{f+\texttt{f}}-d{g}_{f}\|_\epsilon+\|d\texttt{g}_{f+\texttt{f}}\|_\epsilon.$$
The first part tends to 0 when $\|\texttt{f}\|_\epsilon\to 0$ according to \cite{HT}, Lemma 5.1-a with $k=0$, $s=0$, since $dg$ is of class $\mathcal{C}^1$. The second part is uniformly bounded in $\|\texttt{f}\|_\epsilon$ by $\|d\texttt{g}\|_{\mathcal{C}^1}$ by (\ref{IT}). Hence $\|d_1B_{(f+\texttt{f},g+\texttt{g})}-d_1B_{(f,g)}\|\to 0$ when $(\texttt{f},\texttt{g})\to 0$ in $\mathcal{C}^{\epsilon}(I,\B^k)\times\mathcal{C}^2(\B^k,\C)$, and the map $d_1B$ is continuous at point $(f,g)$.

Hence $B$ admits continuous partial differentials $d_1B$ and $d_2B$.
\end{proof}

The differential  $d_1 \Phi (\bm{h},\tilde{r},J_{st})$ is the continuous linear map from $\mathcal{C}^{1,\epsilon}(\bar{\Delta},T^*\C^{n+1})$ to $\mathcal{C}^{1,\epsilon}(\partial \Delta,\R^{2n+2}) \times
 \mathcal C^{\epsilon}(\Delta)$ defined for every $u \in \mathcal{C}^{1,\epsilon}(\bar{\Delta},T^*\C^{n+1})$ by 
$$\begin{array}{llllll}
d_1 \Phi (\bm{h},r,J_{st})(u) = \left(
\begin{matrix}
2 \Re e  (\bar{G} u) \\
\frac{\partial u}{\partial \bar \zeta}
\end{matrix}
\right).
\end{array}$$
We recall that  for $\zeta \in \partial \Delta$, 
$G(\zeta)=\displaystyle \left(\frac{\partial \tilde{r}_i}{\partial\bar{z}_j}(\bm{h}(\zeta))\right)_{i,j}$, and $G$ is smooth since $\tilde{r}$ and ${\bm h}$ are of class $\mathcal{C}^\infty$.

As noticed in \cite{CGS}, it follows easily from the resolution of the classical $\bar \partial$-problem in the unit disc  
that if the linear map $u \mapsto 2 \Re e (\bar G u)$ from $\mathrm{Hol}(\Delta,T^*\C^{n+1})\cap\mathcal{C}^{1,\epsilon}(\bar{\Delta},T^*\C^{n+1})$ to $\mathcal{C}^{1,\epsilon}(\partial \Delta,\R^{2n+2})$ 
 is onto  then $d_1 \Phi (\bm{h},\tilde{r},J_{st})$ is onto. In that case, according to the implicit function Theorem, 
the set of $J$-stationary discs glued to a hypersurface $\Gamma$ is locally a $(2n+2+k)$-dimensional manifold providing 
that $\Gamma$ is a small $\mathcal{C}^4$-perturbation of the hyperquadric and that $J$ is a small $\mathcal{C}^3$-perturbation of the standard structure, where
 $k$ is the dimension of the kernel of $u \mapsto 2 \Re e (\bar G u)$. 

In \cite{Glob2} (Theorems 3.1 and 6.1), it is proved that if the partial indices of the totally real fibration $\mathcal{N}Q$ along 
$\bm{h}|_{\partial \Delta}$ are nonnegative then the linear map 
$u \mapsto 2 \Re e (\bar G u)$ from $\mathcal{C}^{\epsilon}(\bar{\Delta})$ to $\mathcal{C}^{\epsilon}(\partial{\Delta})$ is onto and has a $k$ dimensional kernel, where $k$ is the Maslov index of $\mathcal{N}Q$ along 
$\bm{h}|_{\partial \Delta}$. Partial indices and the Maslov index of $\mathcal{N}Q$ along 
$\bm{h}|_{\partial \Delta}$ have been already computed in subsection 3.2 (see Lemma \ref{lemparind} and Lemma \ref{lemmasind}). Moreover, usual properties of the Cauchy transform give that if $v\in\mathcal{C}^{1,\epsilon}(\partial{\Delta})$, then the solution $u$ satisfying $2 \Re e (\bar G u)=v$ constructed in \cite{Glob2} is actually in $\mathcal{C}^{1,\epsilon}(\bar{\Delta})$, thus the map $u \mapsto 2 \Re e (\bar G u)$ from $\mathcal{C}^{1,\epsilon}(\bar{\Delta})$ to $\mathcal{C}^{1,\epsilon}(\partial{\Delta})$ is still onto (and still has a $k$ dimensional kernel). We obtain:

\begin{theo}\label{descriptionpseudoDisques}
Let $Q=\{r=0\}$ where $r(z)=\Re e z_0-\tsp\bar{z}_\alpha A z_\alpha$ and $A$ is an invertible Hermitian $(n\times n)$ matrix. Fix $\bm{h}\in\mathscr{S}^*(Q)$ and an open ball $\B\subset\C^{n+1}$ such that $\bm{h}(\partial\Delta)\subset\B\times\B$. 
Then for any $0<\epsilon<1$, there exist an open 
neighborhood $V$ of $r$ in $\mathcal{C}^4(\B)$, $\lambda>0$   
and $\delta>0$ such that for any $\rho \in V$ and  any almost complex structure $J$  satisfying $\|J-J_{st}\|_{\mathcal{C}^3(\B)}< \lambda$, 
the set of discs $\bm{f} \in \mathscr{S}_J(\Gamma^{\rho})$  such that 
$\parallel \bm{f} -\bm{h} \parallel_{\mathcal{C}^{1,\epsilon}(\bar{\Delta})} \leq \delta$ forms a $(4n+4)$-real parameter family. 
\end{theo}

In the above Theorem, it is important to notice that according to Proposition \ref{propco2}, since $\rho$ and $J$ are small 
deformations of $r$ and $J_{st}$, the conormal bundle is $\mathcal{N}_J\Gamma^{\rho}$ is a totally real deformation of 
$\mathcal{N}Q$. We also obtain an almost complex analogue of Theorem \ref{paramperturb} and Corollary \ref{propopbis}:
\begin{theo}\label{paramperturb2}
Assume that the conditions of Theorem \ref{descriptionpseudoDisques} are satisfied.
 Then
$$\mathscr{S}_{\bm{h},\delta,J}^*(\Gamma^{{\rho}}):=\{\bm{f}=(f,g) \in\mathscr{S}^*_J(\Gamma^{{\rho}})\ |
\|\bm{f}-\bm{h}\|_{\mathcal{C}^{1,\epsilon}(\partial\Delta)}<\delta\}$$
forms a (2n+2)-real parameter family. 
Moreover, if $\tsp\overline{h_\alpha(0)}A h_\alpha(0)\not=0$, one can reduce the neighborhoods in order to get:
\begin{enumerate}[i)]
\item the discs $\bm{f}\in\mathscr{S}^*_{\bm{h},\delta,J}(\Gamma^\rho)$  satisfy $\bm{f}(\overline{\Delta})\subset\B\times\B$;
\item the map $\bm{f}\mapsto f(0)$ is a diffeomorphism of class $\mathcal{C}^1$ from $\mathscr{S}^*_{\bm{h},\delta,J}(\Gamma^\rho)$ onto its image;
\item  the map $\bm{f}\mapsto\bm{f}'(1)$ defined on $\mathscr{S}^*_{\bm{h},\delta,J}(\Gamma^\rho)$ is one-to-one.
\item if $h(0)\in\Omega=\{v=(v_0,v_\alpha)\in\C^{n+1}\ |\ \tsp\overline{v}_\alpha A v_\alpha\not=0\}$ then there exists 
an open neighborhood $O$ of $h(0)$ satisfying $O\subset\Omega$ and 
$O\subset\{f(0)\ |\ \bm{f}\in\mathscr{S}_{\bm{h},\delta/2,J}^*(\Gamma^{{\rho}})\}$ 
as soon as $\|\rho-r\|_{\mathcal{C}^4(\B)}<\varepsilon$.
\end{enumerate}
\end{theo}

This is basically obtained by following the proof of  Theorem \ref{paramperturb} and  
introducing the $J$-holomorphy equation as a second member of the studied maps as it was done in (\ref{map}).

\subsection{2-jet determination in the almost complex case}
Let $J$ be an almost complex structure of class $\mathcal{C}^{3}$ 
defined on $\R^{2n+2}$, and $\Gamma$ be a real hypersurface of class $\mathcal{C}^4$. Assume $\Gamma$ is $J$-Levi non-degenerate at $p$, and that $F$ is a $(J,J)$-biholomorphism such that $F(\Gamma)\subset\Gamma$. As in the standard case, we use the description of stationary discs given in Theorem \ref{paramperturb2} to get a finite jet determination result. In the following, we only point out the modifications we need in order to adapt the proof given in the standard case.

We first reduce to $p=0$ and $J(0)=J_{st}$. Then, following the proof of Theorem \ref{thmjet} in Section \ref{preuve}, we can assume that
the hypersurface $\Gamma$ is given in a
neighborhood of the origin  by the defining function
\begin{equation*}
\rho(z)=x_0-\tsp \overline{z}_\alpha Az_\alpha+b_0y_0^2+\sum_{j=1}^n(b_jz_j+\bar{b}_j\bar{z}_j)y_0+O(|(y_0,z_\alpha)|^3)
\end{equation*}
which is the normal form (\ref{normalform}). Since this equation is obtained after a {\em holomorphic} change of variable, we still have $J(0)=J_{st}$ and we can write
$$J(z)=\left(\begin{array}{cc}J_{st}^{1}+A(z)&B(z)\\C(z)&J_{st}^{n}+D(z)\end{array}\right)$$
where $J_{st}^{1}$ (resp. $J_{st}^{n}$) denotes the standard structure of $\R^2$ (resp. $\R^{2n}$) and $A$, $B$, $C$, $D$ are of class $\mathcal{C}^3$ and vanish at 0. In particular the $(2\times 2n)$ matrix $B$ has the following form
$$B(z)=\sum_{k=0}^{n}(B_{2k}x_k+B_{2k+1}y_k)+O(\|z\|^2).$$  

Fix some point $z\in\Omega=\{(\gamma\tsp\bar{v}Av,v)\ |\ v\in\C^n,\ \tsp\bar{v}Av\not=0,\ \Re
e(\gamma)>1\}$ and let  $\bm{h}=(h,\zeta h^*)\in\mathscr{S}^*Q$ be the unique holomorphic disc 
satisfying $h(0)=z$. Let $\B \subset \C^{n+1}$ be an open ball centered at the origin,  such that 
$\bm{h}(\overline{\Delta})\subset\B\times\B$. 
By Theorem \ref{paramperturb2}, there exist $\varepsilon>0$, $\lambda>0$ and  $\delta>0$ such that if
$\|{\rho}-r\|_{\mathcal{C}^4(\B)}<\varepsilon$ and $\|{J}-J_{st}\|_{\mathcal{C}^3(\B)}<\lambda$ then the set 
$$\mathscr{S}_{\bm{h},\delta,{J}}^*(\Gamma^{{\rho}})=\{\bm{f}\in\mathscr{S}^*_{{J}}(\Gamma^{{\rho}})\ |
\|\bm{f}-\bm{h}\|_{\mathcal{C}^{1,\epsilon}(\bar \Delta)}<\delta\}$$
is a $(2n+2)$-parameter family parametrized by $\bm{f}\mapsto f(0)$, and the map
$$\begin{array}{rcl}
\mathscr{S}_{\bm{h},\delta,{J}}^*(\Gamma^{{\rho}})&\to&\C^{n+1}\times\C^{n+1}\\
\bm{f}&\mapsto&\bm{f}'(1) 
\end{array}$$
is well-defined and one-to-one.

As in Section \ref{preuve}, we use an inhomogeneous dilation in order to reduce to a sufficiently small perturbation of the model case $\Gamma=Q$, $J=J_{st}$. We underline the fact that ``sufficiently'' is given by Theorem \ref{paramperturb2} and thus depends on the choice of $\bm{h}$ (that is, of our choice of point $z$). Let us use the same notations as in the standard case: 
$\Lambda_t:(z_0,z_\alpha)\mapsto (t^2z_0,tz_\alpha)$ is the inhomogeneous dilation, 
$\Gamma_t:=\Lambda_t^{-1}(\Gamma)$ and $\rho_t:=\frac{1}{t^2}\rho\circ\Lambda_t$ defines the dilated hypersurface, and 
 $F_t:=\Lambda_t^{-1}\circ F\circ\Lambda_t$ is the dilated map. We also set 
$J_t:=\Lambda_t^{-1}\circ  J(\Lambda_t)\circ \Lambda_t$ the dilated almost complex structure:
\begin{equation}\label{Jt}
J_t=\left(\begin{array}{cc}J_{st}^{1}+A\circ\Lambda_t&\frac{1}{t}\,B\circ\Lambda_t\\t\,C\circ\Lambda_t&J_{st}^{n}+D\circ\Lambda_t\end{array}\right).
\end{equation}
Notice that $F_t=\Lambda_t^{-1}\circ F\circ \Lambda_t$ is a $(J_t,J_t)$-biholomorphism.

We already know (see Lemmas \ref{cst-t} and \ref{cst-K}) that there exists some $t_0>0$ such that for any $0<t<t_0$, $\|\rho-r\|_{\mathcal{C}^4(\B)}<\varepsilon$ and ${F_t}_*\bm{f}$ is in $\mathscr{S}_{\bm{h},\delta,{J_t}}^*(\Gamma^{{\rho_t}})$ for any $\bm{f}\in\mathscr{S}_{\bm{h},\delta,{J_t}}^*(\Gamma^{{\rho_t}})$. 
Hence it only remains to check the convergence of the dilated almost complex structures $J_t$ as $t$ tends to $0$. 
 According to (\ref{Jt}), the sequence $(J_t)$ converges with respect to the $\mathcal{C}^3$-topology on any compact subset of $\R^{2n+2}$  
to the model structure $J_{0}$ defined by 
$$J_0(z):=\left(\begin{array}{cc}J_{st}^{1}&b(z_{\alpha})\\0&J_{st}^{n}\end{array}\right),$$
where
$b(z_{\alpha})=\sum_{k=1}^{n}(B_{2k}x_k+B_{2k+1}y_k)$ is the linear part of $B$ in the last $n$ coordinates. It follows that one can choose $t_0>0$ so that for any $0<t<t_0$, the condition
$$\|J_t-J_0\|_{\mathcal{C}^3(\B)}<\lambda$$
is also satisfied.

\bigskip 

{\em Assume $2n>4$.} In case the model structure $J_0$ is sufficiently close to $J_{st}$, 
the same proof than in the standard case gives that $F_t$ is equal to the identity on the open set $O$, and thus $F$ is equal to the identity on some open set. We conclude by using the following result, which is certainly well-known:

\begin{prop}
Let $F$ and $G$ be two pseudoholomorphic maps between almost complex manifolds $(M,J)$ and $(M',J')$. Assume that $(M,J)$ is of class $\mathcal{C}^r$, $r>1$ and $(M,J)$ is of class $\mathcal{C}^{r'}$, $r'>0$. If $M$ is connected by path and $F$ and $G$ coincide on an open set, then they coincide everywhere.
\end{prop}
\begin{proof}
Assume that $F$ and $G$ coincide on some neighborhood of some point $a$. Pick $b$ and some continuous path between $a$ and $b$. According to Proposition A1 in \cite{IR}, for any $z$ on this path, there exists some $R(z)>0$ such that for any points $\alpha,\beta\in\B(z,R(z))$, one can construct a $J$-holomorphic disc $u$ satisfying $u(0)=\alpha$, $u(\frac{1}{2})=\beta$. We take a finite covering of the path by such balls $\B(z_0,R(z_0)),\hdots,\B(z_k,R(z_k))$, with $z_0=a$ and $z_k=b$. We can assume that two successive balls intersect, and that $F$ and $G$ coincide on $\B(z_0,R(z_0))$.

Pick $\alpha_0\in\B(z_0,R(z_0))\cap\B(z_1,R(z_1))$. For any $\beta\in\B(z_1,R(z_1))$, there exists 
a $J$-holomorphic disc $u$ passing trough $\alpha_0$ and $\beta$. The discs $F\circ u$ and $G\circ u$ coincide on the nonempty open set $u^{-1}(\B(z_0,R(z_0))\cap\B(z_1,R(z_1)))$.
It then follows from the unique continuation property of pseudoholomorphic discs 
(see Proposition 3.2.1 in \cite{Sikorav}) that $F\circ u$ coincide with $G\circ u$ on $\Delta$. Thus $F(\beta)=G(\beta)$ and so $F$ and $G$ coincide on $\B(z_1,R(z_1))$. Gradually, this gives that
 $F$ and $G$ coincide on $\B(z_k,R(z_k))$ and so $F(b)=G(b)$.
\end{proof}

Therefore we obtain:
\begin{prop}
Let $J$ and $J'$ be two almost complex structures  of class $\mathcal{C}^{3}$ 
defined on $\R^{2n+2}$. Let $\Gamma$, $\Gamma'$ be two real hypersurfaces of class $\mathcal{C}^4$. 
Assume $\Gamma$ is $J$-Levi non-degenerate at $p$, and that $J$ is sufficiently close to the standard structure $J_{st}$ with respect to $\mathcal{C}^1$-topology. Then the germs at $p$ of $(J,J')$-biholomorphisms $F$ such that $F(\Gamma)=\Gamma'$ are uniquely determined by their 2-jet at $p$.
\end{prop}

Of course, this result only holds under the assumption that the model structure $J_0$ is close enough to $ J_{st}$. 
In the general case, we would  obtain a small perturbation of the hyperquadric $Q$ equipped with a model structure $J_0$ (see \cite{GS} for the study of these special almost complex structures). So our method would need to determine the $J_0$-stationary discs glued to $Q$, and to compute the partial indices in this case. 

\bigskip

{\em Assume $2n=4$}. The situation is rather different and we do not need to consider only perturbations of the 
standard structure. We use the following
\begin{lem} 
If $n=2$,  we can always assume that $J_0=J_{st}$.
\end{lem}
\begin{proof}
In $\R^4$ one can choose coordinates centered at the origin such that the structure $J$ has the form $J=\left(\begin{array}{cc}J_{st}^{1}+A&0\\0&J_{st}^{1}+D\end{array}\right)$, where $A$ and $D$ are two $(2\times 2)$ matrices vanishing at 0, and where the equation of $\Gamma$ is given by  
\begin{equation*}
x_0+\Re e (cz_1^2)-a|z_1|^2+b_0y_0^2+(b_1z_1+\bar{b}_1\bar{z}_1)y_0+O(|(y_0,z_1)|^3)=0.
\end{equation*}

After the local change of coordinates $z_0=z_0'-cz_1^2,\
z_1=z_1'$, the hypersurface $\Gamma$ is given by the defining function
\begin{equation*}
\rho(z)=x_0-a|z_1|^2+b_0y_0^2+(b_1z_1+\bar{b}_1\bar{z}_1)y_0+O(|(y_0,z_1)|^3)
\end{equation*}
and the structure, still denoted by $J$, has the form  $J=\left(\begin{array}{cc}J_{st}^{1}+A&B\\0&J_{st}^{1}+D\end{array}\right)$, where 
$B(z)=O(|z_1|\|z\|)$. We thus obtain that $J_t=\Lambda_t^{-1}\circ  J(\Lambda_t)\circ \Lambda_t$ converges to the standard structure $J_{st}$.
\end{proof}

Hence this proves Theorem \ref{theopC} and we get a finite jet determination result in a four dimensional almost complex manifold. 

\subsection{Boundary version of the uniqueness Theorem} 
In the almost complex setting, the following analogue of H. Cartan's uniqueness Theorem is due to \cite{Lee}:
let $F:M\to M$ be a pseudoholomorphic map, where $(M,J)$ is a smooth almost complex manifold assumed to be connected and Kobayashi hyperbolic; if $F(p)=p$ and $dF_p=\mathrm{Id}$, then $F$ is the identity mapping.

The proof of Theorem \ref{theopC} actually gives a boundary version of this result. As in the standard case, we only used the fact that the pseudoholomorphic map $(F,dF^{-1})$, even defined locally only on one side of $\Gamma$, extends $\mathcal{C}^1$-smoothly to $\Gamma$. This is the case if $F$ is a pseudo-biholomorphism, or more generally a proper pseudoholomorphic map, between two bounded strictly pseudoconvex regions with $\mathcal{C}^4$ boundary. So we obtain: 

\begin{cor}
Let $J$ and $J'$ be two almost complex structures of class $\mathcal{C}^{3}$ 
defined in $\R^{4}$. 
Let $\Omega$ (resp. $\Omega'$) be a bounded strictly $J$-convex region (resp. a bounded strictly $J'$-convex region) in  $\R^{4}$. Assume the boundaries of these domains are of class $\mathcal{C}^4$ and let $p\in\partial\Omega$. If $F_1$ and $F_2$ are two proper $(J,J')$-holomorphic maps from $\Omega$ to $\Omega'$ with the same 2-jet at $p$, they coincide.
\end{cor}
\noindent We recall that  a strictly $J$-convex region $\Omega$  of class $\mathcal{C}^4$  is a domain admitting a {\em global} 
defining function of class $\mathcal{C}^4$ whose Levi form is positive definite in a neighborhood of $\overline{\Omega}$.

\vskip 0,5cm
{\small
\noindent Florian Bertrand\\
Department of Mathematics, University of Vienna, Nordbergstrasse 15, Vienna, 1090, Austria\\
{\sl E-mail address}: florian.bertrand@univie.ac.at\\
\\
L\'ea Blanc-Centi \\
Laboratoire Paul Painlev\'e, Universit\'e Lille 1, 59655 Villeneuve d'Ascq C\'edex, France\\
{\sl E-mail address}: lea.blanc-centi@math.univ-lille1.fr\\
}

\end{document}